\documentclass[a4paper, 12pt]{amsart}

\usepackage{amsmath}
\usepackage{amssymb}
\usepackage{ascmac}
\usepackage{amsthm}

\makeatletter

\@addtoreset{equation}{section}
\makeatother 

\theoremstyle{plain}
\newtheorem{thm}{Theorem}[section]
\newtheorem*{thm*}{Theorem}
\newtheorem{prop}[thm]{Proposition}
\newtheorem{lem}[thm]{Lemma}
\newtheorem{cor}[thm]{Corollary}

\theoremstyle{definition}

\theoremstyle{remark}
\newtheorem{rem}[thm]{Remark}

\renewcommand{\epsilon}{\varepsilon}

\newcommand{\ric}{\operatorname{Ric}}
\newcommand{\Div}{\operatorname{div}}

\newcommand{\tr}{\operatorname{tr}}

\newcommand{\CD}{\operatorname{CD}}

\usepackage{latexsym}

\title[Liouville theorem for heat equation along ancient super Ricci flow]{Liouville theorem for heat equation along\\ ancient super Ricci flow via reduced geometry}

\author{Keita Kunikawa}
\address{Cooperative Faculty of Education, Utsunomiya University, 350 Mine-Machi, Utsunomiya, 321-8505, Japan}
\email{kunikawa@cc.utsunomiya-u.ac.jp}

\author{Yohei Sakurai}
\address{Department of Mathematics, Saitama University, 255 Shimo-Okubo, Sakura-ku, Saitama-City, Saitama, 338-8570, Japan}
\email{ysakurai@rimath.saitama-u.ac.jp}

\date{June 2, 2021}

\begin{document}
\maketitle

\begin{abstract}
The aim of this article is to provide a Liouville theorem for heat equation along ancient super Ricci flow.
We formulate such a Liouville theorem under a growth condition concerning Perelman's reduced distance.
\end{abstract}

\section{Background}

In the present paper,
we establish a Liouville theorem for heat equation along ancient super Ricci flow from Perelman's reduced geometric viewpoint (see Theorem \ref{thm:main result} and Corollary \ref{thm:main cor} below).
In this first section,
we describe the background of our work.

\subsection{Ancient super Ricci flow}\label{sec:Ancient super Ricci flow}
We first recall the notion of ancient super Ricci flow.

Let $(M,g(t))_{t\in I}$ be a smooth manifold equipped with a time-dependent Riemannian metric.
Such a time-dependent Riemannian manifold is called \textit{Ricci flow} if $g(t)$ evolves by
\begin{equation}\label{eq:RF}
\partial_{t}g= -2 \ric.
\end{equation}
The notion of Ricci flow has been established by Hamilton \cite{H1}.
Perelman \cite{P} has vastly developed Hamilton's Ricci flow theory,
and solved the Poincar\'e conjecture.
Since his celebrated work,
the Ricci flow theory has been one of the central objects in geometric analysis.

A supersolution to the Ricci flow equation (\ref{eq:RF}) is called super Ricci flow.
More precisely,
a time-dependent Riemannian manifold $(M,g(t))_{t\in I}$ is called \textit{super Ricci flow} if
\begin{equation*}\label{eq:sRF}
\partial_{t}g\geq -2 \ric.
\end{equation*}
Notice that
static Riemannian manifold of non-negative Ricci curvature is a trivial example of super Ricci flow;
in particular,
super Ricci flow can be viewed as a time-dependent version of manifold of non-negative Ricci curvature.
The notion of super Ricci flow has been introduced by McCann-Topping \cite{MT}.
They have focused on the relation between Ricci flow theory and optimal transport theory,
and characterized super Ricci flow in terms of the monotonicity of Wasserstein distance along heat flow.
After that,
several other characterizations have been investigated via Bochner inequality, Bakry-\'Emery gradient estimate, Poincar\'e inequality, logarithmic Sobolev inequality, and convexity of entropies (see e.g., \cite{HN}, \cite{LL4}, \cite{S}).
Recently,
based on such characterizations,
the notion of super Ricci flow has been extended to time-dependent (non-smooth) metric measure spaces by Sturm \cite{S},
and begun to be studied from metric measure geometric perspective (see e.g., \cite{K}, \cite{KS}, \cite{S}).

A Ricci flow $(M,g(t))_{t\in I}$ is said to be \textit{ancient} when $I=(-\infty,0]$.
Ancient Ricci flow is one of the crucial objects in singular analysis of Ricci flow.
Actually,
it appears as a parabolic rescaling limit at singularity.
In the present paper,
we say that a super Ricci flow $(M,g(t))_{t\in I}$ is \textit{ancient} when $I=(-\infty,0]$.

\subsection{Liouville theorems for ancient solutions to heat equation}
We next recall some previous works on Liouville theorems for ancient solutions to heat equation.

Liouville type properties for harmonic functions on Riemannian manifolds are some of the central topics in geometric analysis since the pioneering work by Yau \cite{Y} (cf. the recent survey \cite{CM} for its history).
The classical Liouville theorem by Yau asserts that
on a complete manifold of non-negative Ricci curvature,
any positive harmonic functions must be constant.

One of the natural research directions is to extend such Liouville type properties to ancient solutions to heat equation.
Here we recall that
on a Riemannian manifold $(M,g)$,
a solution $u:M\times I\to \mathbb{R}$ to the heat equation
\begin{equation*}
\partial_t u=\Delta u
\end{equation*}
is called \textit{ancient} when $I=(-\infty,0]$.
Souplet-Zhang \cite{SZ} have established the following parabolic version of the Yau's Liouville theorem (see Theorem 1.2 in \cite{SZ}):
\begin{thm}[\cite{SZ}]\label{thm:SZ}
Let $(M,g)$ be a complete Riemannian manifold of non-negative Ricci curvature.
Then we have the following:
\begin{enumerate}\setlength{\itemsep}{+0.7mm}
\item Let $u:M\times (-\infty,0]\to (0,\infty)$ be a positive ancient solution to the heat equation.
If
\begin{equation}\label{eq:positive growth2}
u(x,t)=\exp\left[o\left(d(x)+\sqrt{\vert t\vert}\right)\right]
\end{equation}
near infinity, then $u$ must be constant.
Here $d(x)$ denotes the Riemannian distance from a fixed point; \label{enum:positive growth}
\item let $u:M\times (-\infty,0]\to \mathbb{R}$ be an ancient solution to the heat equation.
If
\begin{equation}\label{eq:usual growth2}
u(x,t)=o\left(d(x)+\sqrt{\vert t\vert}\right)
\end{equation}
near infinity, then $u$ is constant. \label{enum:usual growth}
\end{enumerate}
\end{thm}

The growth conditions (\ref{eq:positive growth2}) and (\ref{eq:usual growth2}) are sharp in the spatial direction (see \cite{SZ}).

As stated in Subsection \ref{sec:Ancient super Ricci flow},
super Ricci flow can be regarded as a time-dependent version of manifold of non-negative Ricci curvature.
From this point of view,
it seems to be natural to seek for Liouville theorems for heat equation along ancient super Ricci flow.
More precisely,
for an ancient super Ricci flow $(M,g(t))_{t\in (-\infty,0]}$,
the problem is to find suitable growth conditions for a solution $u:M\times (-\infty,0]\to \mathbb{R}$ to heat equation
\begin{equation*}
\partial_t u=\Delta u
\end{equation*}
such that $u$ must become constant.

In the study of ancient Ricci flow,
we often work on the reverse time parameter
\begin{equation*}
\tau:=-t.
\end{equation*}
On this parameter,
the above problem can be translated as follows:
For a backward super Ricci flow $(M,g(\tau))_{\tau \in [0,\infty)}$,
namely,
\begin{equation}\label{eq:back SRF}
\ric\geq \frac{1}{2}\partial_{\tau}g,
\end{equation}
the problem is to find suitable conditions for a solution $u:M\times [0,\infty)\to \mathbb{R}$ to backward heat equation
\begin{equation}\label{eq:bHE}
(\Delta+\partial_{\tau}) u=0
\end{equation}
such that $u$ must become constant.
Guo-Philipowski-Thalmaier \cite{GPT2} have approached this problem,
and shown a Liouville theorem under a growth condition for entropy.
Let us briefly recall their result.
Let $(M,g(\tau))_{\tau \in [0,\infty)}$ be a complete backward super Ricci flow.
For a fixed $x_0 \in M$,
let $k(x_0,x,\tau)$ stand for the heat kernel of the backward conjugate heat equation
\begin{equation*}
\left(\Delta-\partial_{\tau}-\frac{1}{2}(\tr \partial_{\tau}g)\right)v=0.
\end{equation*}
In \cite{GPT2},
for a positive solution $u:M\times [0,\infty)\to (0,\infty)$ to backward heat equation,
they have defined its \textit{entropy} with respect to the heat kernel measure $k(x_0,x,\tau)m(x)$ by
\begin{equation*}
\mathcal{E}(\tau):=\int_{M}\,(u \log u)(x,\tau)\,k(x_0,x,\tau)\,dm(x).
\end{equation*}
They have shown the following (see Theorem 2 in \cite{GPT2}):
\begin{thm}[\cite{GPT2}]\label{thm:entropy}
Let $(M,g(\tau))_{\tau \in [0,\infty)}$ be a complete backward super Ricci flow.
Let $u:M\times [0,\infty) \to (0,\infty)$ be a positive solution to backward heat equation.
We assume
\begin{align*}
&\int_{M} \Vert \nabla(u \log u) \Vert^2(x,\tau)\,k(x_0,x,\tau)\,dm(x)<\infty,\\
&\int_{M} \left\Vert \nabla  \left( \frac{\Vert \nabla u \Vert^2}{u}  \right) \right\Vert^2(x,\tau)\,k(x_0,x,\tau)\,dm(x)<\infty
\end{align*}
for every $\tau>0$.
If we have
\begin{equation}\label{eq:entropy condition}
\lim_{\tau \to \infty}\frac{\mathcal{E}(\tau)}{\tau}=0,
\end{equation}
then $u$ is constant.
\end{thm}

The linear growth condition (\ref{eq:entropy condition}) is also sharp (see Example 1 in \cite{GPT2}).

\begin{rem}
In \cite{GPT2},
the roles of $t$ and $\tau$ are reversed (cf. Remark 2 in \cite{GPT2}).
\end{rem}

\section{Main results}
In this section,
we present our main result,
and the key ingredient of its proof.

\subsection{Liouville theorems via reduced geometry}\label{sec:Liouville theorems}
We now state our Liouville theorem.

In \cite{P},
Perelman has discovered two kinds of important monotone quantities along Ricci flow.
The one is entropy functionals called \textit{$\mathcal{F}$-functional} and \textit{$\mathcal{W}$-functional},
and the other is an integral quantity called \textit{reduced volume}.
Guo-Philipowski-Thalmaier \cite{GPT2} have introduced the growth condition (\ref{eq:entropy condition}) in Theorem \ref{thm:entropy} inspired by Perelman's entropy functionals (see Remark 1 in \cite{GPT2}).
We aim to produce a reduced geometric counterpart of Theorem \ref{thm:entropy}.
We will prove a Liouville theorem under a growth condition for Perelman's reduced distance.

To state our main result,
we recall some notions on a complete, time-dependent Riemannian manifold $(M,g(\tau))_{\tau \in [0,\infty)}$,
which is not necessarily backward super Ricci flow.
We define
\begin{equation*}\label{eq:hH}
h:=\frac{1}{2}\partial_{\tau}g,\quad H:=\tr h.
\end{equation*}

First,
we briefly recall the notion of reduced distance (we explain contents of this paragraph in more detail in Subsection \ref{sec:Basics of reduced geometry}).
For $(x,\tau)\in M\times (0,\infty)$,
let $L(x,\tau)$ stand for the \textit{$L$-distance} from a space-time base point $(x_0 ,0)$,
which is defined as the infimum of the so-called \textit{$\mathcal{L}$-length} over all curves $\gamma:[0,\tau]\to M$ with $\gamma(0)=x_0$ and $\gamma(\tau)=x$.
Then the \textit{reduced distance $\ell(x,\tau)$} is defined by
\begin{equation*}\label{eq:reduced distance}
\ell(x,\tau):=\frac{1}{2\sqrt{\tau}}L(x,\tau).
\end{equation*}
In the static case of $g(\tau)\equiv g$,
it holds that
\begin{equation}\label{eq:static reduced distance}
\ell(x,\tau)=\frac{d(x)^2}{4\tau}
\end{equation}
for the Riemannian distance $d(x)$ from $x_0$ induced from $g$.
We say that
$(M,g(\tau))_{\tau \in [0,\infty)}$ is \textit{admissible} if
for every $\tau>0$ there is $c_{\tau}\geq 0$ depending only on $\tau$ such that $h \geq -c_{\tau} g$ on $[0,\tau]$.
The admissibility ensures that
the $L$-distance can be achieved by a minimal $\mathcal{L}$-geodesic.

Next,
for a (time-dependent) vector field $V$,
we recall the following \textit{M\"uller quantity} $\mathcal{D}(V)$ (see Definition 1.3 in \cite{M}),
and \textit{trace Harnack quantity} $\mathcal{H}(V)$ (see \cite{H2}, Definition 1.5 in \cite{M}):
\begin{align*}
\mathcal{D}(V)&:=-\partial_{\tau}H-\Delta H-2\Vert h \Vert^2+4\Div h(V)-2g(\nabla H,V)+2\ric(V,V)-2h(V,V),\\
\mathcal{H}(V)&:=-\partial_{\tau} H-\frac{H}{\tau}-2g(\nabla H,V)+2h(V,V).
\end{align*}

\begin{rem}
Along general geometric flow,
several geometric and analytic properties are known to hold under the non-negativity of M\"uller quantity.
For instance,
we possess the monotonicity of $\mathcal{W}$-functional (see Theorem 3.1 in \cite{H} and Theorem 5.2 in \cite{GPT1}),
and that of reduced volume (see Theorem 1.4 in \cite{M} and Theorem 3.4 in \cite{H}).
We also have a differential Harnack inequality of Perelman type (see Corollary 9.3 in \cite{P} and Theorem 1.2 in \cite{CGT}),
that of Cao type (see Theorem 1.3 in \cite{Ca}, Theorem 2.2 in \cite{GH} and Theorem 1.4 in \cite{FZ}),
and that of Cao-Hamilton type (see Theorem 1.1 in \cite{CH}, Theorem 1.4 in \cite{F} and Theorem 3.2 in \cite{GH}).
\end{rem}

Now,
let us consider a backward super Ricci flow $(M,g(\tau))_{\tau \in [0,\infty)}$.
The backward super Ricci flow inequality (\ref{eq:back SRF}) can be written as
\begin{equation}\label{eq:new back SRF}
\ric\geq h.
\end{equation}
The static case of $h=0$ corresponds to the case where $(M,g(\tau))_{\tau \in [0,\infty)}$ is a complete manifold of non-negative Ricci curvature,
and then $H=0$.
The equality case of $h=\ric$ does the case where it is a complete backward Ricci flow,
and then $H$ is equal to the scalar curvature.
Our main result is the following:
\begin{thm}\label{thm:main result}
Let $(M,g(\tau))_{\tau \in [0,\infty)}$ be an admissible, complete backward super Ricci flow.
We assume
\begin{equation}\label{eq:main assumption}
\mathcal{D}(V)\geq 0,\quad \mathcal{H}(V) \geq -\frac{H}{\tau},\quad H\geq 0
\end{equation}
for all vector fields $V$.
Then we have the following:
\begin{enumerate}\setlength{\itemsep}{+0.7mm}
\item Let $u:M\times [0,\infty)\to (0,\infty)$ be a positive solution to backward heat equation.
If
\begin{equation}\label{eq:positive growth}
u(x,\tau)=\exp\left[o\left(\mathfrak{d}(x,\tau)+\sqrt{\tau}\right)\right]
\end{equation}
near infinity, then $u$ is constant.
Here $\mathfrak{d}(x,\tau)$ is defined by
\begin{equation*}
\mathfrak{d}(x,\tau):=\sqrt{4\tau\,\ell(x,\tau)};
\end{equation*}
\item let $u:M\times [0,\infty)\to \mathbb{R}$ be a solution to backward heat equation.
If
\begin{equation}\label{eq:usual growth}
u(x,\tau)=o\left(\mathfrak{d}(x,\tau)+\sqrt{\tau}\right)
\end{equation}
near infinity, then $u$ is constant. 
\end{enumerate}
\end{thm}
Regarding the assumption for $\mathcal{H}(V)$ in (\ref{eq:main assumption}),
the right hand side and the second term in the definition of $\mathcal{H}(V)$ cancel out each other.
Also,
the non-negativity of $H$ guarantees that of $\ell(x,\tau)$;
in particular,
$\mathfrak{d}(x,\tau)$ is well-defined.

\begin{rem}\label{rem:static Liouville}
In the static case of $h=0$,
Theorem \ref{thm:main result} is reduced to Theorem \ref{thm:SZ}.
Indeed,
we have $H=0,\,\mathcal{D}(V)=\ric(V,V),\,\mathcal{H}(V)=0$,
and hence the admissibility and the assumption (\ref{eq:main assumption}) automatically hold.
Furthermore,
the equality (\ref{eq:static reduced distance}) yields $\mathfrak{d}(x,\tau)=d(x)$.
\end{rem}

\begin{rem}\label{rem:division}
For later convenience,
we divide $\mathcal{D}(V)$ into two parts.
We set
\begin{align}\label{eq:Muller main part}
\mathcal{D}_{0}(V)&:=-\partial_{\tau}H-\Delta H-2\Vert h \Vert^2+4\Div h(V)-2g(\nabla H,V),\\ \label{eq:Muller sRF part}
\mathcal{R}(V)&:=\ric(V,V)-h(V,V)
\end{align}
such that $\mathcal{D}(V)=\mathcal{D}_{0}(V)+2\mathcal{R}(V)$ (cf. Definition 1 in \cite{I}).
The backward super Ricci flow inequality (\ref{eq:new back SRF}) implies $\mathcal{R}(V)\geq 0$.
In particular,
if we have $\mathcal{D}_{0}(V)\geq 0$,
then the assumption $\mathcal{D}(V)\geq 0$ in (\ref{eq:main assumption}) is satisfied.
\end{rem}

We now observe the equality case of $h=\ric$.
Let $(M,g(\tau))_{\tau \in [0,\infty)}$ be a complete backward Ricci flow.
Then we see $\mathcal{D}(V)=0$
since the evolution formula of scalar curvature along Ricci flow tells us that
the sum of the first three terms of $\mathcal{D}_{0}(V)$ vanishes (see e.g., Corollary 4.20 in \cite{AH}),
the contracted second Bianchi identity says that
the remaining part of $\mathcal{D}_{0}(V)$ also does (see e.g., (2.16) in \cite{AH}),
and the equality condition $h=\ric$ leads to $\mathcal{R}(V)=0$.
Additionally,
if we suppose the non-negativity of curvature operator,
then the admissibility and $H\geq 0$ hold.
Moreover,
if we assume its boundedness,
then the Hamilton's trace Harnack inequality for ancient Ricci flow yields the assumption for $\mathcal{H}(V)$ in (\ref{eq:main assumption}) (see Corollary 1.2 in \cite{H2}, and also Theorem \ref{thm:KtrHarnack} and Corollary \ref{cor:ancient KtrHarnack} below).
Thus,
from Theorem \ref{thm:main result} we conclude:
\begin{cor}\label{thm:main cor}
Let $(M,g(\tau))_{\tau \in [0,\infty)}$ be a complete backward Ricci flow with bounded, non-negative curvature operator.
Then we have:
\begin{enumerate}\setlength{\itemsep}{+0.7mm}
\item Let $u:M\times [0,\infty)\to (0,\infty)$ be a positive solution to backward heat equation.
If
\begin{equation*}
u(x,\tau)=\exp\left[o\left(\mathfrak{d}(x,\tau)+\sqrt{\tau}\right)\right]
\end{equation*}
near infinity, then $u$ is constant;
\item let $u:M\times [0,\infty)\to \mathbb{R}$ be a solution to backward heat equation.
If
\begin{equation*}
u(x,\tau)=o\left(\mathfrak{d}(x,\tau)+\sqrt{\tau}\right)
\end{equation*}
near infinity, then $u$ is constant.
\end{enumerate}
\end{cor}

\begin{rem}
From the viewpoint of the works by Bailesteanu-Cao-Pulemotov \cite{BCP}, Zhang \cite{Z}, and Ecker-Knopf-Ni-Topping \cite{EKNT},
it seems to be reasonable to consider growth conditions not for Riemannian distance $d(x,\tau)$ induced from $g(\tau)$ but for $\mathfrak{d}(x,\tau)$
in formulating a Liouville theorem of Souplet-Zhang type for solutions to backward heat equation along backward Ricci flow.
We discuss this topic in Subsection \ref{sec:Discussion}.
\end{rem}

\subsection{Gradient estimates}
We here present the key ingredient of the proof of Theorem \ref{thm:main result}.

For that purpose,
we now recall the notion of $K$-super Ricci flow,
which is a time-dependent version of Riemannian manifold whose Ricci curvature is bounded from below by $K$.
For $K\in \mathbb{R}$,
a time-dependent Riemannian manifold $(M,g(t))_{t\in I}$ is called \textit{$K$-super Ricci flow} if
\begin{equation}\label{eq:KKsRF}
\frac{1}{2}\partial_{t}g+\ric \geq K g.
\end{equation}
Here $0$-super Ricci flow is nothing but super Ricci flow.
When the equality in (\ref{eq:KKsRF}) holds,
$(M,g(t))_{t\in I}$ is called \textit{$K$-Ricci flow}.

\begin{rem}
In the context of metric measure geometry,
the following general notion called $(K,N)$-super Ricci flow has been investigated:
For $K\in \mathbb{R}$ and $N\in [n,\infty]$,
a time-dependent weighted Riemannian manifold $(M,g(t),\phi(t))_{t\in I}$ is called \textit{$(K,N)$-super Ricci flow} if
\begin{equation*}\label{eq:sRF}
\frac{1}{2}\partial_{t}g+\ric^{N}_{\phi} \geq K g,
\end{equation*}
where $n$ is the dimension of $M$,
and $\ric^{N}_{\phi}$ denotes the so-called \textit{$N$-Bakry-\'Emery Ricci curvature} associated with the time-dependent density function $\phi(t)$.
We notice that
$(K,N)$-super Ricci flow is a time-dependent version of weighted Riemannian manifold satisfying the condition that $N$-Bakry-\'Emery Ricci curvature is bounded from below by $K$,
which is equivalent to the so-called \textit{curvature-dimension condition $\CD(K,N)$} in the sense of Sturm \cite{S1}, \cite{S2} and Lott-Villani \cite{LV}.
It seems that
the notion of $(K,N)$-super Ricci flow firstly appeared in the work of Arnaudon-Coulibaly-Thalmaier \cite{ACT} (see Theorem 4.1 (b) in \cite{ACT}),
where they only consider the case of $N=\infty$.
After their work,
it has been studied by Sturm \cite{S}, Kopfer-Sturm \cite{KS}, Kopfer \cite{K},
or in a series of works done by Li-Li \cite{LL1}, \cite{LL2}, \cite{LL3}, \cite{LL4}.
\end{rem}

We now consider a backward $K$-super Ricci flow $(M,g(\tau))_{\tau \in [0,\infty)}$,
namely,
\begin{equation}\label{eq:KsRF}
\ric \geq h+K g.
\end{equation}
The backward $K$-super Ricci flow inequality (\ref{eq:KsRF}) leads to $\mathcal{R}(V)\geq K \Vert V\Vert^2$.
Theorem \ref{thm:main result} is a direct consequence of the following space-only local gradient estimate:
\begin{thm}\label{thm:gradient estimate}
For $K\geq 0$,
let $(M,g(\tau))_{\tau \in [0,\infty)}$ be an $n$-dimensional, admissible, complete backward $(-K)$-super Ricci flow.
We assume 
\begin{equation}\label{eq:Kmain assumption}
\mathcal{D}(V)\geq -2K\left(H+\Vert V\Vert^2   \right),\quad \mathcal{H}(V) \geq -\frac{H}{\tau},\quad H\geq 0
\end{equation}
for all vector fields $V$.
Let $u:M \times [0,\infty) \to (0,\infty)$ stand for a positive solution to backward heat equation.
For $R,T>0$ and $A>0$,
we suppose $u\leq A$ on
\begin{equation*}
Q_{R,T}:=\left\{\,(x,\tau)\in M\times \left(0,T\right] \,\,\, \middle|\,\,\, \mathfrak{d}(x,\tau)\leq R  \,\right\}.
\end{equation*}
Then there exists a positive constant $C_{n}>0$ depending only on $n$ such that on $Q_{R/2,T/4}$,
\begin{equation*}
\frac{\Vert \nabla u \Vert}{u}\leq C_n \left( \frac{1}{R}+\frac{1}{\sqrt{T}}+\sqrt{K}  \right) \left( 1+\log \frac{A}{u}  \right).
\end{equation*}
\end{thm}

In the static case of $h=0$,
Theorem \ref{thm:gradient estimate} has been established by Souplet-Zhang \cite{SZ} (see Theorem 1.1 in \cite{SZ}, and cf. Remark \ref{rem:static Liouville}).
We here recall that
the gradient estimate in \cite{SZ} is a parabolic version of the local gradient estimate of Cheng-Yau \cite{CY} for harmonic functions,
and also inspired by the space-only global gradient estimate by Hamilton \cite{H3}.

\begin{rem}
By the backward $(-K)$-super Ricci flow inequality (\ref{eq:KsRF}),
if we have $\mathcal{D}_{0}(V)\geq -2KH$,
then the assumption for $\mathcal{D}(V)$ in (\ref{eq:Kmain assumption}) holds (cf. Remark \ref{rem:division}).
\end{rem}

In Section \ref{sec:KtrHge},
we examine the trace Harnack inequality for $K$-Ricci flow,
and derive the following result from Theorem \ref{thm:gradient estimate} in the same manner as Corollary \ref{thm:main cor}:
\begin{cor}\label{cor:grad cor}
For $K \geq 0$,
let $(M,g(\tau))_{\tau \in [0,\infty)}$ be an $n$-dimensional,
complete backward $(-K)$-Ricci flow with bounded, non-negative curvature operator.
Let $u:M \times [0,\infty) \to (0,\infty)$ be a positive solution to backward heat equation.
For $R,T>0$ and $A>0$,
we suppose $u\leq A$ on $Q_{R,T}$.
Then there is a positive constant $C_{n}>0$ depending only on $n$ such that on $Q_{R/2,T/4}$,
\begin{equation*}
\frac{\Vert \nabla u \Vert}{u}\leq C_n \left( \frac{1}{R}+\frac{1}{\sqrt{T}}+\sqrt{K}  \right) \left( 1+\log \frac{A}{u}  \right).
\end{equation*}
\end{cor}

In Section \ref{sec:Comparisons with other},
we compare our gradient estimate with other space-only local gradient estimates obtained by Bailesteanu-Cao-Pulemotov \cite{BCP}, Zhang \cite{Z}, Ecker-Knopf-Ni-Topping \cite{EKNT},
and discuss the advantage of our work.

\section{Reduced geometry}\label{sec:Reduced geometry}
To capture the behavior of Ricci flow,
Perelman \cite{P} has introduced the reduced geometry,
which is the comparison geometry on his reduced distance.
This section is devoted to review and study on reduced geometry.
Throughout this subsection,
let $(M,g(\tau))_{\tau \in [0,\infty)}$ stand for an $n$-dimensional, complete time-dependent Riemannian manifold.

\subsection{Basics of reduced geometry}\label{sec:Basics of reduced geometry}
We first review basics of reduced geometry.
The references are \cite{CCGG}, \cite{M}, \cite{P}, \cite{Ye}, \cite{Y1}, \cite{Y2}.
We mainly refer to the works of Yokota \cite{Y1}, \cite{Y2},
which is compatible with our setting.

We start with giving the precise definition of reduced distance.
For a curve $\gamma:[\tau_1,\tau_2]\to M$,
its \textit{$\mathcal{L}$-length} is defined as
\begin{equation*}
\mathcal{L}(\gamma):=\int^{\tau_2}_{\tau_1}\sqrt{\tau}\left( H +\left\Vert \frac{d\gamma}{d\tau} \right\Vert^{2} \right)\,d\tau.
\end{equation*}
Notice that
its critical point over all curves with fixed endpoints is known to be characterized by the following \textit{$\mathcal{L}$-geodesic equation}:
\begin{equation*}\label{geod}
X:=\frac{d\gamma}{d\tau},\quad \nabla_X X - \frac{1}{2}\nabla H+ \frac{1}{2\tau}X + 2h(X) =0.
\end{equation*}
For $(x,\tau)\in M\times (0,\infty)$,
the \textit{$L$-distance} $L(x,\tau)$ and \textit{reduced distance} $\ell(x,\tau)$ from a space-time base point $(x_0,0)$ are defined by
\begin{equation}\label{eq:L and l}
L(x,\tau):=\inf_{\gamma}\mathcal{L}(\gamma),\quad \ell(x,\tau):=\frac{1}{2\sqrt{\tau}}L(x,\tau),
\end{equation}
where the infimum is taken over all curves $\gamma:[0,\tau]\to M$ with $\gamma(0)=x_0$ and $\gamma(\tau)=x$.
A curve is called \textit{minimal $\mathcal{L}$-geodesic} from $(x_0,0)$ to $(x,\tau)$ if it attains the infimum of (\ref{eq:L and l}).

We next discuss the admissibility introduced in Subsection \ref{sec:Liouville theorems}.
Recall that
$(M,g(\tau))_{\tau \in [0,\infty)}$ is called \textit{admissible} if
for each $\tau>0$ there is $c_{\tau}\geq 0$ depending only on $\tau$ such that $h \geq -c_{\tau} g$ on $[0,\tau]$.
Note that $h=\ric$ if it is backward Ricci flow.
Perelman \cite{P} has developed reduced geometry under uniform boundedness of sectional curvature.
Ye \cite{Ye} has shown that
many fundamental parts of Perelman's reduced geometry, including $\mathcal{L}$-geodesic theory, still hold under a lower Ricci curvature bound.
Furthermore,
Yokota \cite{Y1}, \cite{Y2} has pointed out that
for general geometric flow that is not necessarily (backward) Ricci flow,
the admissibility ensures such fundamental properties with the same proof as in \cite{Ye} (cf. Section 2 in \cite{Y1}).

We assume that
$(M,g(\tau))_{\tau \in [0,\infty)}$ is admissible.
Then for every $(x,\tau)\in M\times (0,\infty)$,
there is at least one minimal $\mathcal{L}$-geodesic (see Proposition 2.8 in \cite{Ye}).
Also,
the functions $L(\cdot,\tau)$ and $L(x,\cdot)$ are locally Lipschitz in $(M,g(\tau))$ and $(0,\infty)$,
respectively (see Propositions 2.12, 2.13, and also 2.14 in \cite{Ye});
in particular,
they are differentiable almost everywhere.

Let us state derivative formulas for reduced distance.
In the rest of this section,
we assume that
$\ell$ is smooth at $(\overline{x},\overline{\tau})\in M\times (0,\infty)$.
For such a point,
the minimal $\mathcal{L}$-geodesic is uniquely determined,
and its tangent vector $\overline{X}$ coincides with the gradient of $\ell$ (see Lemma 2.11 in \cite{Ye}).
We possess the following (see Subsection 2.3 in \cite{Y1}):
\begin{prop}\label{prop:derivative formula}
At $(\overline{x},\overline{\tau})$ we have
\begin{align}\label{eq:time ell}
\partial_\tau \ell &=H-\frac{\ell}{\overline{\tau}}+\frac{1}{2\overline{\tau}^{3/2}}\mathcal{K}_{\mathcal{H}},\\ \label{eq:grad ell}
\Vert \nabla \ell \Vert^2&=-H+\frac{\ell}{\overline{\tau}}-\frac{1}{\overline{\tau}^{3/2}}\mathcal{K}_{\mathcal{H}},\\ \label{eq:Laplace ell}
\Delta \ell &\leq -H+\frac{n}{2\overline{\tau}}-\frac{1}{2\overline{\tau}^{3/2}}\mathcal{K}_{\mathcal{H}}-\frac{1}{2\overline{\tau}^{3/2}}\mathcal{K}_{\mathcal{D}},
\end{align}
where
\begin{equation*}
\mathcal{K}_{\mathcal{H}}:=\int^{\overline{\tau}}_{0}\,\tau^{3/2}\,\mathcal{H}(\overline{X})\,d\tau,\quad \mathcal{K}_{\mathcal{D}}:=\int^{\overline{\tau}}_{0}\,\tau^{3/2}\,\mathcal{D}(\overline{X})\,d\tau.
\end{equation*}
\end{prop}

\begin{rem}
The estimates (\ref{eq:time ell}), (\ref{eq:grad ell}) have been stated in \cite{Y1}.
We can check them by the same calculations as that in the proof of (7,88), (7.89) in \cite{CCGG} for backward Ricci flow.
Also,
a similar Laplacian comparison to (\ref{eq:Laplace ell}) has been done in \cite{Y1}.
But in \cite{Y1}, its last term did not appear since he assumed the non-negativity of $\mathcal{D}_{0}(V)$ and $\mathcal{R}(V)$ defined as (\ref{eq:Muller main part}) and (\ref{eq:Muller sRF part}), which implies that of $\mathcal{D}(V)$, at first (see Assumption 2.1 in \cite{Y1}).
One can verify (\ref{eq:Laplace ell}) by carefully following the calculations in the proof of (7.90) in \cite{CCGG}.
\end{rem}

\begin{rem}\label{rem:barrier}
For general geometric flow,
M\"uller \cite{M} has obtained similar inequalities to (\ref{eq:time ell}), (\ref{eq:grad ell}), (\ref{eq:Laplace ell}) under compactness of $M$ (see (5.18) in \cite{M}).
Here
he has studied both forward and backward reduced distances in a unified way,
and the backward one corresponds to our reduced distance.
Moreover,
he has shown that
even if $\ell$ is not smooth at $(\overline{x},\overline{\tau})$,
they hold in the barrier sense (see Lemma 5.3 in \cite{M}).
By using the same barrier function constructed in the proof of Lemma 5.3 in \cite{M},
we can also conclude that
even if $\ell$ is not smooth at $(\overline{x},\overline{\tau})$,
the inequalities (\ref{eq:time ell}), (\ref{eq:grad ell}), (\ref{eq:Laplace ell}) hold in the barrier sense.
More precisely,
there exist a sufficiently small $\delta>0$,
an open neighborhood $U$ of $\overline{x}$ in $M$,
and a smooth upper barrier function $\widehat{\ell}$ at $(\overline{x},\overline{\tau})$ on $U\times (\overline{\tau}-\delta,\overline{\tau}+\delta)$ (i.e., $\widehat{\ell}\geq \ell$, and the equality holds at $(\overline{x},\overline{\tau})$) such that
$\widehat{\ell}$ satisfies (\ref{eq:time ell}), (\ref{eq:grad ell}), (\ref{eq:Laplace ell}) at $(\overline{x},\overline{\tau})$.
\end{rem}

For $(x,\tau)\in M\times (0,\infty)$ we define
\begin{equation*}
\overline{L}(x,\tau):=4\tau \,\ell(x,\tau)=\mathfrak{d}(x,\tau)^2.
\end{equation*}
Summing up (\ref{eq:time ell}) and (\ref{eq:Laplace ell}) yields the following (see Section 2 in \cite{Y2}):
\begin{lem}\label{lem:reduced heat estimate}
At $(\overline{x},\overline{\tau})$ we have
\begin{equation*}
(\Delta+\partial_\tau)\overline{L} \leq 2n-\frac{2}{\sqrt{\overline{\tau}}}\mathcal{K}_{\mathcal{D}}.
\end{equation*}
\end{lem}

\subsection{Key estimates}\label{sec:Estimates}
In this subsection,
we will collect two key estimates for the proof of our main results.
Let $(M,g(\tau))_{\tau \in [0,\infty)}$ stand for an $n$-dimensional, admissible, complete time-dependent Riemannian manifold.
Continuing from the above subsection,
we also assume that
the reduced distance is smooth at $(\overline{x},\overline{\tau})\in M\times (0,\infty)$.
We first note the following:
\begin{lem}\label{lem:reduced heat estimate2}
Let $K\geq 0$.
We assume
\begin{equation*}\label{eq:assume reduced heat estimate2}
\mathcal{D}(V) \geq -2K\left(H + \Vert V \Vert^2\right),\quad H\geq 0
\end{equation*}
for all vector fields $V$.
Then at $(\overline{x},\overline{\tau})$ we have
\begin{equation*}
(\Delta+\partial_\tau)\overline{L} \leq 2n+2K \overline{L}.
\end{equation*}
\end{lem}
\begin{proof}
Lemma \ref{lem:reduced heat estimate} together with the assumption and $\overline{L}=2\sqrt{\overline{\tau}} L$ tells us that
\begin{align*}
(\Delta+\partial_\tau)\overline{L} &\leq 2n-\frac{2}{\sqrt{\overline{\tau}}}\mathcal{K}_{\mathcal{D}}\leq 2n+\frac{4K}{\sqrt{\overline{\tau}}} \int^{\overline{\tau}}_{0}\,\tau^{3/2}\,\left( H+\Vert \overline{X}\Vert^2  \right)\,d\tau\\
                                                    &\leq 2n+4K \sqrt{\overline{\tau}} \int^{\overline{\tau}}_{0}\,\tau^{1/2}\,\left( H+\Vert \overline{X}\Vert^2  \right)\,d\tau=2n+4K \sqrt{\overline{\tau}}L=2n+2K \overline{L}.
\end{align*}
This proves the lemma.
\end{proof}

We also prove that
under the trace Harnack inequality and the non-negativity of $H$,
the norm of the gradient of $\mathfrak{d}$ is bounded by a universal constant (cf. (2.54) in \cite{Ye}).
\begin{lem}\label{lem:reduced gradient estimate}
We assume
\begin{equation}\label{eq:assumption rge}
\mathcal{H}(V) \geq -\frac{H}{\tau},\quad H\geq 0
\end{equation}
for all vector fields $V$.
Then at $(\overline{x},\overline{\tau})$ we have
\begin{equation*}
\Vert \nabla \mathfrak{d} \Vert^2\leq 3.
\end{equation*}
\end{lem}
\begin{proof}
From the trace Harnack inequality in (\ref{eq:assumption rge}) it follows that
\begin{equation}\label{eq:applicationHH}
\mathcal{K}_{\mathcal{H}} \geq -\int^{\overline{\tau}}_{0}\,\sqrt{\tau}\,H\,d\tau\geq -L= -2\sqrt{\overline{\tau}}\,\ell.
\end{equation}
By the non-negativity of $H$,
and by applying (\ref{eq:applicationHH}) to (\ref{eq:grad ell}),
we obtain
\begin{equation*}
\Vert \nabla \ell \Vert^2\leq \Vert \nabla \ell \Vert^2+H=\frac{\ell}{\overline{\tau}}-\frac{1}{\overline{\tau}^{3/2}}\mathcal{K}_{\mathcal{H}}\leq \frac{3 \ell}{\overline{\tau}}.
\end{equation*}
This is equivalent to the desired one.
\end{proof}

\begin{rem}\label{rem:static universal}
This estimate seems to be natural since in the static case of $h=0$, the function $\mathfrak{d}$ coincides with the Riemannian distance function from the fixed point,
whose norm of the gradient is equal to a universal constant $1$ (see Remark \ref{rem:static Liouville}).
\end{rem}

\section{Proof of main results}\label{sec:Proof of main results}
We now prove Theorem \ref{thm:gradient estimate},
and conclude Theorem \ref{thm:main result}.
We give a proof of Theorem \ref{thm:gradient estimate} by using the key estimates obtained in Subsection \ref{sec:Estimates}
along the line of the proof of Souplet-Zhang gradient estimate in \cite{SZ} (for comparison between our method of the proof and that in \cite{SZ}, see Subsection \ref{sec:comp SZ} below).
Throughout this section,
let $(M,g(\tau))_{\tau \in [0,\infty)}$ stand for an $n$-dimensional, admissible, complete time-dependent Riemannian manifold.

\subsection{Backward heat equations}\label{sec:Formulas for backward heat equations}
In the next two subsections,
we will make preparations towards the proof of Theorem \ref{thm:gradient estimate}.
Here,
we examine properties of a positive solution $u:M\times [0,\infty)\to (0,\infty)$ to backward heat equation.
We begin with the following:
\begin{lem}\label{lem:easy}
We set
\begin{equation}\label{eq:log heat}
f:=\log u.
\end{equation}
Then we have
\begin{align}\label{eq:easy1}
(\Delta +\partial_{\tau}) f&=-\Vert \nabla f \Vert^2,\\ \label{eq:easy2}
(\Delta+\partial_{\tau}) \Vert \nabla f \Vert^2& = 2\Vert \nabla^2 f \Vert^2-2g(\nabla \Vert \nabla f \Vert^2,\nabla f)+2\mathcal{R}(\nabla f),
\end{align}
where $\mathcal{R}(\nabla f)$ is defined as $(\ref{eq:Muller sRF part})$.
\end{lem}
\begin{proof}
From direct calculations and backward heat equation (\ref{eq:bHE}),
we deduce
\begin{equation*}
\Delta f=-\frac{\Vert \nabla u \Vert^2}{u^2}+\frac{\Delta u}{u}=-\frac{\Vert \nabla u \Vert^2}{u^2}-\frac{\partial_{\tau}u }{u}=-\Vert \nabla f \Vert^2-\partial_{\tau}f,
\end{equation*}
and hence (\ref{eq:easy1}).
The second one (\ref{eq:easy2}) follows from
\begin{align*}
\partial_{\tau} \Vert \nabla f \Vert^2&=-(\partial_{\tau}g)(\nabla f,\nabla f)+2g(\nabla \partial_{\tau} f,\nabla f)\\
                                                &=-2\ric(\nabla f,\nabla f)-2g(\nabla \Delta f,\nabla f)-2g(\nabla \Vert \nabla f \Vert^2,\nabla f)+2\mathcal{R}(\nabla f)\\
                                                &=-\Delta \Vert \nabla f \Vert^2+2\Vert \nabla^2 f \Vert^2-2g(\nabla \Vert \nabla f \Vert^2,\nabla f)+2\mathcal{R}(\nabla f).
\end{align*}
Here we used a fundamental formula for time-dependent metrics (see e.g., Lemma 4.5 in \cite{AH}),
the equation (\ref{eq:easy1}),
and Bochner formula.
This completes the proof.
\end{proof}

We next show the following:
\begin{lem}\label{lem:simple}
Let $f$ be defined as $(\ref{eq:log heat})$.
We assume $\sup f<1$,
and set
\begin{equation}\label{eq:gradient log heat}
w:=\frac{\Vert \nabla f \Vert^2}{(1-f)^2}.
\end{equation}
Then we have
\begin{equation*}
\left(\Delta+\partial_{\tau}\right)w-\frac{2f g(\nabla w,\nabla f)}{1-f}\geq 2(1-f)w^2+\frac{2\mathcal{R}(\nabla f)}{(1-f)^2}.
\end{equation*}
\end{lem}
\begin{proof}
By straightforward computations,
\begin{equation*}
\nabla w=\frac{\nabla \Vert \nabla f\Vert^2}{(1-f)^2}+\frac{2\Vert \nabla f\Vert^2\nabla f}{(1-f)^3},
\end{equation*}
and hence
\begin{equation}\label{eq:simple1}
g(\nabla w,\nabla f)=\frac{g(\nabla \Vert \nabla f\Vert^2,\nabla f)}{(1-f)^2}+\frac{2\Vert \nabla f\Vert^4}{(1-f)^3}=\frac{g(\nabla \Vert \nabla f\Vert^2,\nabla f)}{(1-f)^2}+2(1-f)w^2.
\end{equation}
Furthermore,
(\ref{eq:simple1}) yields
\begin{align}\label{eq:simple2}
\Delta w&=\frac{\Delta \Vert \nabla f \Vert^2}{(1-f)^2}+\frac{4g(\nabla \Vert \nabla f \Vert^2,\nabla f)}{(1-f)^3}+\frac{6 \Vert \nabla f \Vert^4}{(1-f)^4}+\frac{2\Vert \nabla f\Vert^2 \Delta f}{(1-f)^3}\\ \notag
                     &= \frac{2g(\nabla w,\nabla f)}{1-f}+\frac{2\Vert \nabla f \Vert^4}{(1-f)^4}+\frac{2g(\nabla \Vert \nabla f \Vert^2,\nabla f)}{(1-f)^3}+\frac{\Delta \Vert \nabla f \Vert^2}{(1-f)^2}+\frac{2\Vert \nabla f\Vert^2 \Delta f}{(1-f)^3}.
\end{align}
Likewise,
Lemma \ref{lem:easy} and (\ref{eq:simple1}) imply
\begin{align}\label{eq:simple3}
\partial_{\tau} w&=\frac{\partial_{\tau} \Vert \nabla f \Vert^2}{(1-f)^2}+\frac{2\Vert \nabla f \Vert^2 \partial_{\tau} f}{(1-f)^3}\\ \notag
                         &=-\frac{2\Vert \nabla f \Vert^4}{(1-f)^3}-\frac{2g(\nabla \Vert \nabla f\Vert^2,\nabla f)}{(1-f)^2}
                         +\frac{2\Vert \nabla^2 f \Vert^2}{(1-f)^2}+\frac{2\mathcal{R}(\nabla f)}{(1-f)^2}-\frac{\Delta \Vert \nabla f\Vert^2}{(1-f)^2}-\frac{2\Vert \nabla f \Vert^2 \Delta f}{(1-f)^3}\\ \notag
                         &=-2g(\nabla w,\nabla f)+2(1-f)w^2+\frac{2\Vert \nabla^2 f \Vert^2}{(1-f)^2}+\frac{2\mathcal{R}(\nabla f)}{(1-f)^2}-\frac{\Delta \Vert \nabla f\Vert^2}{(1-f)^2}-\frac{2\Vert \nabla f \Vert^2 \Delta f}{(1-f)^3}.
\end{align}
Combining (\ref{eq:simple2}) and (\ref{eq:simple3}),
we obtain
\begin{equation*}
\left(\Delta+\partial_{\tau}\right)w-\frac{2f \,g(\nabla w,\nabla f)}{1-f} =2(1-f)w^2+\frac{2\mathcal{R}(\nabla f)}{(1-f)^2}+2\mathcal{F},
\end{equation*}
where
\begin{equation}\label{eq:simple4}
\mathcal{F}:=\frac{\Vert \nabla f \Vert^4}{(1-f)^4}  +\frac{\Vert \nabla^2 f \Vert^2}{(1-f)^2}+\frac{g(\nabla \Vert \nabla f \Vert^2,\nabla f)}{(1-f)^3}.
\end{equation}

Now,
it is enough to verify that
$\mathcal{F}$ is non-negative.
For the first two terms of (\ref{eq:simple4}),
\begin{equation}\label{eq:simple5}
\frac{\Vert \nabla f \Vert^4}{(1-f)^4}  +\frac{\Vert \nabla^2 f \Vert^2}{(1-f)^2}\geq \frac{2 \Vert \nabla^2 f\Vert \Vert \nabla f \Vert^2}{(1-f)^3}\geq \frac{\Vert \nabla \Vert \nabla f \Vert^2 \Vert \, \Vert \nabla f\Vert}{(1-f)^3}
\end{equation}
by the inequality of arithmetic-geometric means,
and the Kato inequality.
For the third term,
the Cauchy-Schwarz inequality yields
\begin{equation}\label{eq:simple6}
\frac{g(\nabla \Vert \nabla f \Vert^2,\nabla f)}{(1-f)^3}\geq -\frac{\Vert \nabla \Vert \nabla f \Vert^2 \Vert \, \Vert \nabla f\Vert}{(1-f)^3}.
\end{equation}
From (\ref{eq:simple5}) and (\ref{eq:simple6}),
it follows that
$\mathcal{F}$ is non-negative,
and we complete the proof.
\end{proof}

\subsection{Cut-off arguments}\label{sec:Cut-off arguments}
This subsection is devoted to the cut-off argument.
We recall the following elementary fact:
\begin{lem}\label{lem:cutoff}
Let $R,T>0,\,\alpha \in (0,1)$.
Then there is a smooth function $\psi:[0,\infty)\times [0,\infty)\to [0,1]$,
and a constant $C_\alpha>0$ depending only on $\alpha$ such that the following hold:
\begin{enumerate}\setlength{\itemsep}{3pt}
\item $\psi\equiv 1$ on $[0,R/2]\times [0,T/4]$, and $\psi\equiv 0$ on $[R,\infty)\times [T/2,\infty)$;
\item $\partial_r \psi \leq 0$ on $[0,\infty)\times [0,\infty)$, and $\partial_r \psi \equiv 0$ on $[0,R/2]\times [0,\infty)$;
\item we have
\begin{equation*}
         \frac{\vert \partial_r \psi \vert}{\psi^\alpha}\leq \frac{C_{\alpha}}{R},\quad \frac{\vert \partial^2_{r}\psi\vert}{\psi^\alpha}\leq \frac{C_{\alpha}}{R^2},\quad \frac{\vert \partial_{\tau} \psi\vert}{\psi^{1/2}}\leq \frac{C}{T},
\end{equation*}
where $C>0$ is a universal constant.
\end{enumerate}
\end{lem}

Having Lemma \ref{lem:cutoff} at hand,
we obtain the following assertion:
\begin{prop}\label{prop:difficult}
Let $K\geq 0$.
We assume
\begin{equation}\label{eq:cutoff assumption}
\mathcal{R}(V)\geq -K\Vert V \Vert^2,\quad \mathcal{D}(V)\geq -2K\left(H+\Vert V \Vert^2 \right),\quad \mathcal{H}(V) \geq -\frac{H}{\tau},\quad H\geq 0
\end{equation}
for all vector fields $V$.
Let $u:M \times [0,\infty) \to (0,\infty)$ denote a positive solution to backward heat equation.
For $R,T>0$,
we assume $u\leq 1$ on $Q_{R,T}$.
We define $f$ and $w$ as $(\ref{eq:log heat})$ and $(\ref{eq:gradient log heat})$ on $Q_{R,T}$,
respectively.
We also take a function $\psi:[0,\infty)\times [0,\infty)\to [0,1]$ in Lemma \ref{lem:cutoff} with $\alpha=3/4$,
and define
\begin{equation}\label{eq:cutoff function}
\psi(x,\tau):=\psi(\mathfrak{d}(x,\tau),\tau).
\end{equation}
Then we have
\begin{equation}\label{eq:difficult}
(\psi w)^2\leq \frac{\overline{C}_n}{R^4}+\frac{\widetilde{C}_1}{T^2}+\widetilde{C}_2 K^2+2\Phi
\end{equation}
at every point in $Q_{R,T}$ such that the reduced distance is smooth,
where for the universal constants $C_{3/4},C>0$ given in Lemma \ref{lem:cutoff},
we put
\begin{align}\label{eq:dimension epsilon}
\overline{C}_{n}&:=24C^2_{3/4}\left(n^2+\frac{9}{4}+\frac{657}{64}C^2_{3/4}\right),\quad \widetilde{C}_1:=6C^2,\quad \widetilde{C}_2:=24\left(1+\frac{C^2_{3/4}}{4}\right), \\  \label{eq:maximal constant}
\Phi&:=(\Delta+\partial_{\tau})(\psi w)-\frac{2g\left( \nabla\psi, \nabla(\psi w) \right)}{\psi}-\frac{2f g(\nabla(\psi w),\nabla f)}{1-f}.
\end{align}
\end{prop}
\begin{proof}
First,
note that
$u\leq 1$ implies $f\leq 0$,
and hence $w$ is well-defined.
By direct computations and Lemma \ref{lem:simple},
\begin{align*}
\Phi&=\psi \left(\Delta+\partial_{\tau} \right)w-\frac{2\psi f g(\nabla w,\nabla f)}{1-f}+w\left(\Delta+\partial_{\tau}\right)\psi-\frac{2w\Vert \nabla \psi \Vert^2}{\psi}-\frac{2w f g(\nabla \psi,\nabla f)}{1-f}\\
                  &\geq 2(1-f)\psi w^2+\frac{2\psi\mathcal{R}(\nabla f)}{(1-f)^2}+w\left(\Delta+\partial_{\tau} \right)\psi-\frac{2w\Vert \nabla \psi \Vert^2}{\psi}-\frac{2wf g(\nabla \psi,\nabla f)}{1-f}.
\end{align*}
Therefore,
\begin{equation}\label{eq:main difficult}
2(1-f)\psi w^2\leq \Psi_1+\Psi_2+\Psi_3+\Psi_4+\Phi
\end{equation}
for
\begin{equation}\label{eq:estimate constant}
\Psi_1:=-\frac{2\psi \mathcal{R}(\nabla f)}{(1-f)^2},\quad \Psi_2:=-w \left(\Delta+\partial_{\tau}\right)\psi,\quad \Psi_3:=\frac{2w\Vert \nabla \psi \Vert^2}{\psi},\quad \Psi_4:=\frac{2w f g(\nabla \psi,\nabla f)}{1-f}.
\end{equation}

We present upper estimates of $\Psi_1, \Psi_2, \Psi_3,\Psi_4$.
The following Young inequality plays an essential role in the estimates:
For all $p,q\in (1,\infty)$ with $p^{-1}+q^{-1}=1$, $a,b\geq 0$, and $\epsilon>0$,
\begin{equation}\label{eq:Young}
ab\leq \frac{\epsilon a^p}{p}+\frac{b^q}{\epsilon^{q/p} q}.
\end{equation} 
The inequality
\begin{equation}\label{eq:grad est cutoff}
\frac{\Vert \nabla \psi \Vert^2}{\psi^{3/2}}\leq \frac{3 C^2_{3/4}}{R^2}
\end{equation}
is also useful,
which follows from Lemmas \ref{lem:reduced gradient estimate} and \ref{lem:cutoff}.
We first derive an upper bound of $\Psi_1$.
By the assumption for $\mathcal{R}(V)$,
the Young inequality (\ref{eq:Young}) with $p,q=2$,
and $\psi\leq 1$,
\begin{equation}\label{eq:estimate of C_1}
\Psi_1=-\frac{2\psi \mathcal{R}(\nabla f)}{(1-f)^2}\leq 2K \psi w\leq \epsilon \psi^2 w^2+\frac{K^2}{\epsilon}\leq \epsilon \psi w^2+\frac{K^2}{\epsilon}.
\end{equation}
We next show an upper bound of $\Psi_2$.
By direct computations,
we have
\begin{align}\label{eq:different point}
\Psi_2&=-w \left(\Delta+\partial_{\tau}\right)\psi=-w  \left(           \partial_r \psi(\Delta+\partial_{\tau})\mathfrak{d}+\partial^2_r \psi \Vert \nabla \mathfrak{d} \Vert^2+\partial_{\tau} \psi \right)\\ \notag
                  &=-w  \left[           \partial_r \psi   \left(  \frac{1}{2\mathfrak{d}}(\Delta+\partial_{\tau})\overline{L}-\frac{\Vert \nabla \overline{L} \Vert^2}{4\mathfrak{d}^3}  \right)+\partial^2_r \psi \Vert \nabla \mathfrak{d} \Vert^2+\partial_{\tau} \psi \right]\\ \notag
                  &=\frac{w \vert \partial_r \psi \vert}{2\mathfrak{d}}(\Delta+\partial_{\tau})\overline{L}-w \vert \partial_r \psi \vert \frac{\Vert \nabla \overline{L} \Vert^2}{4\mathfrak{d}^3}-w \,\partial^2_r \psi \Vert \nabla \mathfrak{d} \Vert^2-w\,\partial_{\tau} \psi\\ \label{eq:new different point}
                  & \leq \frac{w \vert \partial_r \psi \vert}{2\mathfrak{d}}(\Delta+\partial_{\tau})\overline{L}+w \vert \partial^2_r \psi\vert \Vert \nabla \mathfrak{d} \Vert^2+w \,\vert \partial_{\tau} \psi \vert.
\end{align}
Using Lemmas \ref{lem:reduced heat estimate2}, \ref{lem:reduced gradient estimate} and $\overline{L}=\mathfrak{d}^2$,
we obtain
\begin{align}\label{eq:different point2}
\Psi_2&\leq n \frac{w \vert \partial_r \psi \vert}{\mathfrak{d}}+K w \vert \partial_r \psi \vert \mathfrak{d}+3w \vert \partial^2_r \psi \vert+w \, \vert \partial_{\tau} \psi \vert \\ \notag
                       &\leq \frac{2n}{R} w \vert \partial_r \psi \vert+KR w \vert \partial_r \psi \vert+3w \vert \partial^2_r \psi \vert+w \,\vert \partial_{\tau} \psi \vert,
\end{align}
where in the second inequality,
we used the fact that
$\partial_r \psi$ vanishes on $[0,R/2]\times [0,\infty)$.
By the Young inequality (\ref{eq:Young}) with $p,q=2$,
Lemma \ref{lem:cutoff},
and $\psi\leq 1$,
it holds that
\begin{align}\label{eq:estimate of A}
\Psi_2&\leq \left( \epsilon  \psi w^2+\frac{n^2}{R^2}\frac{ \vert \partial_r \psi \vert^2}{\epsilon  \psi}   \right)
+\left( \epsilon \psi w^2+\frac{K^2 R^2}{4}\frac{ \vert \partial_r \psi \vert^2}{\epsilon \psi}   \right)\\ \notag
&\qquad \qquad \qquad \qquad \qquad \,\,\,+\left( \epsilon \psi w^2+\frac{9}{4}\frac{ \vert \partial^2_r \psi \vert^2}{\epsilon \psi}   \right)
+\left( \epsilon \psi w^2+\frac{1}{4}\frac{ \vert \partial_{\tau}  \psi\vert^2}{\epsilon \psi}   \right) \\ \notag
&\leq 4\epsilon  \psi w^2+\frac{C^2_{3/4}}{\epsilon}\left(n^2+\frac{9}{4}\right)\frac{ \psi^{1/2}}{R^4}+\frac{C^2}{4\epsilon}\frac{1}{T^2}+\frac{C^2_{3/4}}{4\epsilon}K^2\psi^{1/2}\\ \label{eq:new estimate of A}
&\leq 4\epsilon  \psi w^2+\frac{C^2_{3/4}}{\epsilon}\left(n^2+\frac{9}{4}\right)\frac{1}{R^4}+\frac{C^2}{4\epsilon}\frac{1}{T^2}+\frac{C^2_{3/4}}{4\epsilon}K^2.
\end{align}
We next examine an upper bound of $\Psi_3$.
From the Young inequality (\ref{eq:Young}) with $p,q=2$,
and (\ref{eq:grad est cutoff}),
we deduce
\begin{equation}\label{eq:estimate of B}
\Psi_3= \frac{2w\Vert \nabla \psi \Vert^2}{\psi} \leq \epsilon \psi w^2+\frac{\Vert \nabla \psi \Vert^4}{\epsilon \psi^3}\leq \epsilon \psi w^2+\frac{9C^4_{3/4}}{\epsilon}\frac{1}{R^4}.
\end{equation}
We finally investigate $\Psi_4$.
The Cauchy-Schwarz inequality,
the Young inequality (\ref{eq:Young}) with $p=4/3,q=4,\epsilon=4/3$,
and (\ref{eq:grad est cutoff}) tell us that
\begin{align}\label{eq:estimate of C}
\Psi_4&=\frac{2w f g(\nabla \psi,\nabla f)}{1-f}\leq \frac{2w\vert f\vert \Vert \nabla \psi \Vert \Vert \nabla f \Vert}{1-f}=2 w^{3/2} \vert f\vert  \Vert \nabla \psi \Vert\\ \notag
                  &\leq (1-f)\psi w^2+\frac{27}{16} \frac{f^4}{(1-f)^3}  \frac{\Vert \nabla \psi \Vert^4}{\psi^3}\leq (1-f)\psi w^2+\frac{243C^4_{3/4}}{16} \,\frac{f^4}{(1-f)^3}\frac{1}{R^4}.
\end{align}

Combining (\ref{eq:estimate of C_1}), (\ref{eq:estimate of A}), (\ref{eq:estimate of B}), (\ref{eq:estimate of C}) with (\ref{eq:main difficult}),
we conclude
\begin{align*}
(1-f)\psi w^2 \leq 6\epsilon  \psi w^2&+\frac{C^2_{3/4}}{\epsilon}\left(n^2+\frac{9}{4}+9C^2_{3/4}+  \frac{243\epsilon C^2_{3/4}}{16} \,\frac{f^4}{(1-f)^3}   \right)\frac{1}{R^4}+\frac{C^2}{4\epsilon}\frac{1}{T^2}\\
                                                        &+\frac{1}{\epsilon}\left(1+\frac{C^2_{3/4}}{4}\right)K^2+\Phi.
\end{align*}
We divide the both sides by $1-f$.
Since $1/(1-f)\leq 1$ and $f/(1-f)\leq 1$,
we see
\begin{equation*}\label{eq:pre conclusion}
(1-6\epsilon)\psi w^2 \leq \frac{C^2_{3/4}}{\epsilon}\left(n^2+\frac{9}{4}+9C^2_{3/4}+\frac{243\epsilon C^2_{3/4}}{16}\right)\frac{1}{R^4}+\frac{C^2}{4\epsilon}\frac{1}{T^2}+\frac{1}{\epsilon}\left(1+\frac{C^2_{3/4}}{4}\right)K^2+\Phi.
\end{equation*}
By letting $\epsilon \to 1/12$,
we obtain
\begin{equation*}
\psi w^2 \leq \frac{\overline{C}_n}{R^4}+\frac{\widetilde{C}_1}{T^2}+\widetilde{C}_2 K^2+2\Phi.
\end{equation*}
Since $(\psi w)^2\leq \psi w^2$,
we arrive at the desired one (\ref{eq:difficult}).
\end{proof}

\subsection{Proof of Theorems \ref{thm:main result} and \ref{thm:gradient estimate}}\label{sec:Proof of main theorems}

We are now in a position to prove Theorem \ref{thm:gradient estimate}.
\begin{proof}[Proof of Theorem \ref{thm:gradient estimate}]
For $K\geq 0$,
let $(M,g(\tau))_{\tau \in [0,\infty)}$ be backward $(-K)$-super Ricci flow satisfying $(\ref{eq:Kmain assumption})$ for all vector fields $V$,
and let $u:M\times [0,\infty)\to (0,\infty)$ denote a positive solution to backward heat equation.
As noticed in \cite{SZ},
it suffices to show the desired estimate for $A=1$.
For $R,T>0$,
we assume $u\leq 1$ on $Q_{R,T}$.
Define functions $f,w$ and $\psi$ as $(\ref{eq:log heat}), (\ref{eq:gradient log heat})$ and (\ref{eq:cutoff function}),
respectively.
For $\theta>0$
we define a compact subset $Q_{R,T,\theta}$ of $Q_{R,T}$ by
\begin{equation*}
Q_{R,T,\theta}:=\{(x,\tau)\in Q_{R,T} \mid \tau\in [\theta,T]\}.
\end{equation*}
Hereafter,
we fix a small $\theta\in (0,T/4)$.
Also,
we take a maximum point $(\overline{x},\overline{\tau})$ of $\psi w$ in $Q_{R,T,\theta}$.

We first discuss the case where the reduced distance is smooth at $(\overline{x},\overline{\tau})$.
Due to Proposition \ref{prop:difficult},
we obtain
\begin{equation}\label{eq:using difficult}
(\psi w)^2\leq c_n \left(\frac{1}{R^4}+\frac{1}{T^2}+K^2 \right)+2\Phi
\end{equation}
at $(\overline{x},\overline{\tau})$ for
\begin{equation*}
c_{n}:=\max \left\{\overline{C}_{n},\widetilde{C}_1,\widetilde{C}_2 \right\},
\end{equation*}
where $\overline{C}_{n},\widetilde{C}_1,\widetilde{C}_2>0$ and $\Phi$ are defined as $(\ref{eq:dimension epsilon})$ and (\ref{eq:maximal constant}),
respectively.
Notice that
the assumption for $\mathcal{R}(V)$ in (\ref{eq:cutoff assumption}) is satisfied in virtue of the backward $(-K)$-super Ricci flow inequality (\ref{eq:KsRF}).
On the other hand,
since $(\overline{x},\overline{\tau})$ is a maximum point,
it holds that
\begin{equation*}
\Delta(\psi w)\leq 0,\quad \partial_{\tau} (\psi w)\leq 0,\quad \nabla (\psi w)=0
\end{equation*}
at $(\overline{x},\overline{\tau})$;
in particular, $\Phi(\overline{x},\overline{\tau})\leq 0$.
Therefore,
(\ref{eq:using difficult}) implies
\begin{equation*}\label{eq:smooth conclusion}
(\psi w)(x,\tau) \leq (\psi w)(\overline{x},\overline{\tau})\leq c^{1/2}_{n}\left(\frac{1}{R^4}+\frac{1}{T^2}+K^2     \right)^{1/2}\leq c^{1/2}_{n} \left(\frac{1}{R^2}+\frac{1}{T}+K\right)
\end{equation*}
for all $(x,\tau)\in Q_{R,T,\theta}$.

We observe the non-smooth case.
We employ the barrier argument stated in Remark \ref{rem:barrier}.
By the same method as in the proof of Lemma 5.3 in \cite{M},
there are a sufficiently small $\delta>0$,
a small open neighborhood $U$ of $\overline{x}$ in $M$,
and a smooth upper barrier function $\widehat{\ell}$ of the reduced distance $\ell$ at $(\overline{x},\overline{\tau})$ on $U\times (\overline{\tau}-\delta,\overline{\tau}+\delta)$ such that $\widehat{\ell}$ satisfies (\ref{eq:time ell}), (\ref{eq:grad ell}), (\ref{eq:Laplace ell}) at $(\overline{x},\overline{\tau})$.
As analogues of $\mathfrak{d}(x,\tau)$ and $\psi(x,\tau)$,
we define
\begin{equation*}\label{eq:barrier cutoff function}
\widehat{\mathfrak{d}}(x,\tau):=\sqrt{4\tau\,\widehat{\ell}(x,\tau)},\quad \widehat{\psi}(x,\tau):=\psi\left(\widehat{\mathfrak{d}}(x,\tau),\tau \right),
\end{equation*}
where $\psi$ is a function in Lemma \ref{lem:cutoff} with $\alpha=3/4$.
Remark that
$\widehat{\psi}$ is a smooth lower barrier of $\widehat{\psi}$ at $(\overline{x},\overline{\tau})$ (i.e., $\widehat{\psi}\leq \psi$, and the equality holds at $(\overline{x},\overline{\tau})$);
in particular, $(\overline{x},\overline{\tau})$ is a maximum point of $\widehat{\psi}w$ on $(U\times (\overline{\tau}-\delta,\overline{\tau}+\delta)) \cap Q_{R,T,\theta}$.
It follows that
\begin{equation}\label{eq:max relation}
\Delta(\widehat{\psi} w)\leq 0,\quad \partial_{\tau} (\widehat{\psi} w)\leq 0,\quad \nabla (\widehat{\psi} w)=0
\end{equation}
at $(\bar{x},\bar{\tau})$.
Having (\ref{eq:max relation}) at hand,
by repeating the calculations in Lemmas \ref{lem:reduced heat estimate}, \ref{lem:reduced heat estimate2}, \ref{lem:reduced gradient estimate}, and Proposition \ref{prop:difficult},
we can verify
\begin{equation*}\label{eq:nonsmooth conclusion}
(\psi w)(x,\tau) \leq (\psi w)(\bar{x},\bar{\tau})=(\widehat{\psi} w)(\bar{x},\bar{\tau}) \leq c^{1/2}_{n} \left(\frac{1}{R^2}+\frac{1}{T}+K\right)
\end{equation*}
for all $(x,\tau)\in Q_{R,T,\theta}$.
This is the same conclusion as in the smooth case.

In both cases,
by $\psi \equiv 1$ on $Q_{R/2,T/4,\theta}$,
and by the definition of $w$ and $f$,
\begin{equation*}
\frac{\Vert \nabla u \Vert}{u} \leq c^{1/4}_{n} \left(\frac{1}{R}+\frac{1}{\sqrt{T}}+\sqrt{K}\right)\left( 1+\log \frac{1}{u} \right)
\end{equation*}
on $Q_{R/2,T/4,\theta}$.
Thus,
by letting $\theta\to 0$,
we complete the proof of Theorem \ref{thm:gradient estimate}.
\end{proof}

Let us conclude Theorem \ref{thm:main result}.
\begin{proof}[Proof of Theorem \ref{thm:main result}]
Let $(M,g(\tau))_{\tau \in [0,\infty)}$ be backward super Ricci flow satisfying $(\ref{eq:main assumption})$ for all vector fields $V$.
In each statement,
it is enough to show that
the gradient of $u$ vanishes at each point in $M\times (0,\infty)$.
We fix $(x,\tau)\in M\times (0,\infty)$.

Let us show the first statement.
Let $u:M\times [0,\infty)\to (0,\infty)$ denote a positive solution to backward heat equation.
For $R>0$
we put 
\begin{equation*}
A_{R}:=\sup_{Q_{R,R^2}}u.
\end{equation*}
The growth condition (\ref{eq:positive growth}) says $\log \left(A_{R}+1\right)=o(R)$ as $R\to \infty$.
We see $(x,\tau)\in Q_{R/2,R^2/4}$ for every sufficiently large $R>0$,
and fix such one.
Applying Theorem \ref{thm:gradient estimate} with $K=0$ to a function $u+1$ on $Q_{R,R^2}$ leads us to
\begin{equation*}
\frac{\Vert \nabla u \Vert}{u+1}\leq \frac{2C_n}{R} \left( 1+\log \frac{A_{R}+1}{u+1}  \right)\leq \frac{2C_n}{R} \left( 1+\log \left(A_{R}+1\right)  \right)=\frac{2C_n}{R} \left( 1+o(R)  \right)
\end{equation*}
at $(x,\tau)$.
Letting $R\to \infty$,
we arrive at the desired conclusion.

We next prove the second statement.
Let $u:M\times [0,\infty)\to \mathbb{R}$ be a solution to backward heat equation.
For $R>0$
we set 
\begin{equation*}
\overline{A}_{R}:=\sup_{Q_{R,R^2}} \vert u\vert.
\end{equation*}
The growth condition (\ref{eq:usual growth}) tells us $\overline{A}_{R}=o(R)$ as $R\to \infty$.
We now fix a sufficiently large $R>0$ such that $(x,\tau)\in Q_{R/2,R^2/4}$,
and apply Theorem \ref{thm:gradient estimate} to $u+2\overline{A}_{R}$ on $Q_{R,R^2}$.
Since
\begin{equation*}
\overline{A}_{R}\leq u+2\overline{A}_{R}\leq 3\overline{A}_{R}
\end{equation*}
on $Q_{R,R^2}$,
we have
\begin{equation*}
\Vert \nabla u \Vert \leq \frac{2C_n}{R} \left( 1+\log \frac{3\overline{A}_{R}}{u+2\overline{A}_{R}}  \right)\left(u+2\overline{A}_{R}\right)\leq 6\left( 1+\log 3  \right)C_n \frac{o(R)}{R}
\end{equation*}
at $(x,\tau)$.
By letting $R\to \infty$,
we complete the proof of Theorem \ref{thm:main result}.
\end{proof}

\section{Trace Harnack inequalities and gradient estimates for $K$-Ricci flow}\label{sec:KtrHge}
In this section,
we prove Corollary \ref{cor:grad cor}.
The proof is done by extending the Hamilton's trace Harnack inequality for Ricci flow to $K$-Ricci flow.

\subsection{Trace Harnack inequalities}\label{sec:Trace Harnack inequalities}
In the present subsection,
we study the trace Harnack inequality for $K$-Ricci flow.
To do so,
we prepare several key tools.
For $K\neq 0$,
we define a monotone increasing, bijective function $\sigma:J\to \mathbb{R}$ by
\begin{equation*}
\sigma(s):=-\frac{1}{2K}\log (1-2Ks),
\end{equation*}
where
\begin{equation*}
J:=\begin{cases}
             (-\infty,1/2K) & \text{if $K>0$}, \\
              (1/2K,\infty) & \text{if $K<0$}.
           \end{cases}
\end{equation*}
This function satisfies
\begin{equation}\label{eq:basic parametrization}
\sigma'(s)=e^{2K \sigma(s)}.
\end{equation}
Furthermore,
its inverse function $\sigma^{-1}:\mathbb{R}\to J$ is given by
\begin{equation}\label{eq:scaling inverse}
\sigma^{-1}(t)=\frac{1-e^{-2Kt}}{2K}
\end{equation}
with $\sigma^{-1}(0)=0$.
Our proof is based on the following observation (cf. Theorem 1.11 in \cite{KS}):
\begin{lem}\label{lem:scaling argument}
Let $K\neq 0$,
and let $(M,g(t))_{t \in [0,\mathcal{T})}$ be a time-dependent Riemannian manifold.
For $s \in [0,\sigma^{-1}(\mathcal{T}))$ we define
\begin{equation}\label{lem:scaling metric}
\overline{g}(s):=e^{-2K \sigma(s)}g(\sigma(s)).
\end{equation}
Then the following are equivalent:
\begin{enumerate}\setlength{\itemsep}{+1.0mm}
\item $(M,g(t))_{t \in [0,\mathcal{T})}$ is $K$-Ricci flow;
\item $(M,\overline{g}(s))_{s \in [0,\sigma^{-1}(\mathcal{T}))}$ is Ricci flow.
\end{enumerate}
\end{lem}
\begin{proof}
The property (\ref{eq:basic parametrization}) yields
\begin{equation*}
\partial_s \overline{g}=\sigma'(s)e^{-2K \sigma(s)}(-2Kg+\partial_t g)=-2Kg+\partial_t g.
\end{equation*}
We arrive at the desired assertion since the Ricci curvature is invariant under multiplication of Riemannian metric by positive constants.
\end{proof}

The following is the main result of this subsection:
\begin{thm}[\cite{H2}]\label{thm:KtrHarnack}
For $K \in \mathbb{R}$,
let $(M,g(t))_{t \in [0,\mathcal{T})}$ denote a complete $K$-Ricci flow with bounded, non-negative curvature operator.
Then we have
\begin{equation*}
\partial_t S+\frac{2KS}{1-e^{-2Kt}}-2g(\nabla S,V)+2\ric(V,V)\geq 0
\end{equation*}
for all vector fields $V$ and $t>0$,
where $S$ denotes the scalar curvature.
In the case of $K=0$,
we interpret $2K/(1-e^{-2Kt})$ in the second term as the limit $1/t$.
\end{thm}
\begin{proof}
When $K=0$,
this theorem is nothing but the Hamilton's trace Harnack inequality for Ricci flow (see Corollary 1.2 in \cite{H2}).
We now investigate the case of $K\neq 0$.
By virtue of Lemma \ref{lem:scaling argument},
we can construct an associated Ricci flow $(M,\overline{g}(s))_{s \in [0,\sigma^{-1}(\mathcal{T}))}$ determined by (\ref{lem:scaling metric}),
which also has bounded, non-negative curvature operator.
The idea is to apply the Hamilton's trace Harnack inequality to this Ricci flow.
By doing that,
\begin{equation}\label{eq:scaling trace harnack}
\partial_s \overline{S}+\frac{\overline{S}}{s}-2\overline{g}(\overline{\nabla}\, \overline{S},\overline{V})+2\overline{\ric}(\overline{V},\overline{V})\geq 0
\end{equation}
for all vector fields $\overline{V}$ and $s>0$,
where $\overline{\nabla},\overline{\ric},\overline{S}$ are the gradient, Ricci curvature, scalar curvature induced from $\overline{g}(s)$,
respectively.
Let us translate (\ref{eq:scaling trace harnack}) into the language of $g(t)$.
One can verify
\begin{align*}
\overline{\ric}(\overline{V},\overline{V})&=\ric(\overline{V},\overline{V}),\quad \overline{S}=e^{2K \sigma(s)}S,\quad \overline{g}(\overline{\nabla} \,\overline{S},\overline{V})=e^{2K \sigma(s)}g(\nabla S,\overline{V}),\\
\partial_s \overline{S}&=\sigma'(s)e^{2K \sigma(s)}(2KS+\partial_t S)=e^{4K \sigma(s)}(2KS+\partial_t S).
\end{align*}
Substituting these equations into (\ref{eq:scaling trace harnack}),
we obtain
\begin{equation*}
e^{4Kt}(2KS+\partial_t S)+\frac{e^{2Kt}S}{\sigma^{-1}(t)}-2e^{2Kt}g(\nabla S,\overline{V})+2\ric(\overline{V},\overline{V})\geq 0
\end{equation*}
for all vector fields $\overline{V}$ and $t>0$.
In particular,
(\ref{eq:scaling inverse}) leads to
\begin{equation}\label{eq:pre KtH}
\partial_t S+\frac{2KS}{1-e^{-2Kt}}-2g(\nabla S,e^{-2Kt}\overline{V})+2\ric(e^{-2Kt}\overline{V},e^{-2Kt}\overline{V})\geq 0.
\end{equation}
For a given vector field $V$,
we complete the proof by applying (\ref{eq:pre KtH}) to $\overline{V}:=e^{2Kt}V$.
\end{proof}

We can derive the following trace Harnack inequality for ancient $K$-Ricci flow:
\begin{cor}[\cite{H2}]\label{cor:ancient KtrHarnack}
For $K \in \mathbb{R}$,
let $(M,g(t))_{t \in (-\infty,0]}$ be a complete ancient $K$-Ricci flow with bounded, non-negative curvature operator.
Then for all vector fields $V$,
we have:
\begin{enumerate}\setlength{\itemsep}{+1.0mm}
\item If $K> 0$, then
\begin{equation*}
\partial_t S+2KS-2g(\nabla S,V)+2\ric(V,V)\geq 0;
\end{equation*}
\item if $K\leq 0$, then
\begin{equation}\label{eq:negative ancient KtrHarnack}
\partial_t S-2g(\nabla S,V)+2\ric(V,V)\geq 0.
\end{equation}
\end{enumerate}
\end{cor}
\begin{proof}
The desired inequalities follow from the standard argument.
For a negative $t_{0}<0$,
by applying Theorem \ref{thm:KtrHarnack} to a parallel translation $(M,\widetilde{g}(s))_{s \in [0,-t_0)}$ defined as $\widetilde{g}(s):=g(s+t_0)$,
we see that
\begin{equation}\label{eq:pre Kancient}
\partial_t S+\frac{2KS}{1-e^{-2K(t-t_0)}}-2g(\nabla S,V)+2\ric(V,V)\geq 0
\end{equation}
for all vector fields $V$ and $t\in (t_0,0)$.
Now,
for a given time $t\in (-\infty,0)$,
we take $t_0<0$ such that $t_{0}<t$.
Letting $t_0\to -\infty$ in (\ref{eq:pre Kancient}),
we arrive at the desired claim.
\end{proof}

\subsection{Proof of Corollary \ref{cor:grad cor}}\label{sec:Proof of grad cor}
The aim of this subsection is to give a proof of Corollary \ref{cor:grad cor}.
Besides Corollary \ref{cor:ancient KtrHarnack},
we will use the following:
\begin{lem}\label{lem:KFRDV}
For $K\geq 0$,
let $(M,g(\tau))_{\tau \in [0,\infty)}$ be a backward $(-K)$-Ricci flow.
Then for all vector fields $V$, we have
\begin{equation*}
\mathcal{D}(V)=-2K\left(H+\Vert V\Vert^2   \right).
\end{equation*}
\end{lem}
\begin{proof}
The backward $(-K)$-Ricci flow equation can be written as
\begin{equation}\label{eq:-KRFeq}
h=\ric+K g;
\end{equation}
in particular,
for $\mathcal{R}(V)$ defined by $(\ref{eq:Muller sRF part})$,
\begin{equation}\label{eq:modified -KRFeq}
\mathcal{R}(V)=-K\Vert V \Vert^2. 
\end{equation}
On the other hand,
by taking the trace of (\ref{eq:-KRFeq}),
we see
\begin{equation}\label{eq:tr -KRFeq}
H=S+n K
\end{equation}
for the scalar curvature $S$,
where $n$ denotes the dimension of $M$.
Let us substitute (\ref{eq:tr -KRFeq}) into the definition of $\mathcal{D}_{0}(V)$ defined as $(\ref{eq:Muller main part})$.
By the contracted second Bianchi identity,
\begin{equation}\label{eq:KD0}
\mathcal{D}_{0}(V)=-\partial_{\tau} S-\Delta S-2 \Vert \ric \Vert^2-4 K S-2n K^2.
\end{equation}
Also,
we possess the following evolution formula for $S$ (see e.g., Proposition 4.10 in \cite{AH}):
\begin{equation}\label{eq:Kscal}
\partial_{\tau}S=-\Delta S-2 \Vert \ric \Vert^2-2K S.
\end{equation}
Combining (\ref{eq:KD0}) and (\ref{eq:Kscal}),
we obtain
\begin{equation}\label{eq:cplt KD0}
\mathcal{D}_{0}(V)=-2KS-2nK^2=-2KH.
\end{equation}
From (\ref{eq:modified -KRFeq}) and (\ref{eq:cplt KD0}),
we conclude the desired formula.
\end{proof}

We are now in a position to prove Corollary \ref{cor:grad cor}.

\begin{proof}[Proof of Corollary \ref{cor:grad cor}]
For $K \geq 0$,
let $(M,g(\tau))_{\tau \in [0,\infty)}$ be an $n$-dimensional,
complete backward $(-K)$-Ricci flow with bounded, non-negative curvature operator.
It suffices to check that
$(M,g(\tau))_{\tau \in [0,\infty)}$ satisfies the assumption in Theorem \ref{thm:gradient estimate}.

By the backward $(-K)$-Ricci flow equation,
for the scalar curvature $S$,
we have
\begin{equation*}
h=\ric+Kg,\quad H=S+nK.
\end{equation*}
In particular,
by the non-negativity of curvature operator,
the admissibility and the non-negativity of $H$ in (\ref{eq:Kmain assumption}) are satisfied.
Lemma \ref{lem:KFRDV} yields the assumption for $\mathcal{D}(V)$ in (\ref{eq:Kmain assumption}).
Due to (\ref{eq:negative ancient KtrHarnack}) in Corollary \ref{cor:ancient KtrHarnack} and $K\geq 0$,
\begin{equation*}
\mathcal{H}(V)+\frac{H}{\tau}=-\partial_{\tau} S-2g(\nabla S,V)+2\ric(V,V)+2K\Vert V\Vert^2\geq 0,
\end{equation*}
which is the assumption for $\mathcal{H}(V)$ in (\ref{eq:Kmain assumption}).
We complete the proof of Corollary \ref{cor:grad cor}.
\end{proof}

\section{Comparisons with other space-only local gradient estimates}\label{sec:Comparisons with other}
In this section,
we compare Theorem \ref{thm:gradient estimate} with other space-only local gradient estimates.
We aim to make our contribution clear.

\subsection{Souplet-Zhang gradient estimate}\label{sec:comp SZ}
We first compare our method of the proof with that of Souplet-Zhang gradient estimate in \cite{SZ}.

As stated in Section \ref{sec:Proof of main results},
we have proved Theorem \ref{thm:gradient estimate} along the line of the proof of Souplet-Zhang gradient estimate.
On the other hand,
there exists only one technically different part.
That is the upper estimate of $\Psi_2$ demonstrated in (\ref{eq:different point}), (\ref{eq:different point2}) and (\ref{eq:estimate of A}).

In \cite{SZ},
they have obtained such an upper estimate by giving upper bounds of
\begin{equation}\label{eq:SZpoint1}
(\Delta+\partial_{\tau})\mathfrak{d}
\end{equation}
and
\begin{equation}\label{eq:SZpoint2}
\Vert \nabla \mathfrak{d} \Vert^2
\end{equation}
at the stage of (\ref{eq:different point}).
In the static case,
$\mathfrak{d}$ coincides with the Riemannian distance function from the fixed point (see Remark \ref{rem:static Liouville}).
Then they have estimated (\ref{eq:SZpoint1}) by use of the Laplacian comparison in Riemannian geometry,
and the fact that its time derivative vanishes.
For (\ref{eq:SZpoint2}),
they did nothing in fact since it is identical to a universal constant $1$.

In our case,
instead of (\ref{eq:SZpoint1}),
we provide an upper bound of
\begin{equation*}\label{eq:heat reduced distance}
(\Delta+\partial_{\tau})\overline{L}
\end{equation*}
in (\ref{eq:new different point}) by using Lemma \ref{lem:reduced heat estimate2},
which is a combination of the Laplacian comparison (\ref{eq:Laplace ell}) in reduced geometry,
and the time derivative formula (\ref{eq:time ell}) for reduced distance.
Moreover,
we could give a universal upper bound of (\ref{eq:SZpoint2}) in (\ref{eq:new different point}) by Lemma \ref{lem:reduced gradient estimate} (cf. Remark \ref{rem:static universal}).

\subsection{Bailesteanu-Cao-Pulemotov gradient estimate}\label{sec:Bailesteanu-Cao-Pulemotov gradient estimate}
We attempt to compare Theorem \ref{thm:gradient estimate} with other space-only local gradient estimates by Bailesteanu-Cao-Pulemotov \cite{BCP}, Zhang \cite{Z}, and Ecker-Knopf-Ni-Topping \cite{EKNT}.
To make it easier,
in the next three subsections,
we translate their results into our setting and notation.

In this subsection,
we focus on the work of Bailesteanu-Cao-Pulemotov \cite{BCP},
which is most closely related to our work.
In \cite{BCP},
they have produced a space-only local gradient estimate for positive solutions to heat equation along Ricci flow over positive time interval (see Theorem 2.2 in \cite{BCP}).
In our notation,
we can show the following gradient estimate of Bailesteanu-Cao-Pulemotov type by inserting the key techniques in \cite{BCP} into the proof of Theorem \ref{thm:gradient estimate}:
\begin{thm}[\cite{BCP}]\label{thm:BCP}
Let $(M,g(\tau))_{\tau \in [0,\infty)}$ be an $n$-dimensional, complete backward super Ricci flow.
For $K\geq 0$,
we assume 
\begin{equation}\label{eq:bounded Ricci}
\vert \ric \vert\leq Kg.
\end{equation}
Let $u:M \times [0,\infty) \to (0,\infty)$ be a positive solution to backward heat equation.
For $R,T>0$ and $A>0$,
we suppose $u\leq A$ on
\begin{equation*}
B_{R,T}:=\left\{\,(x,\tau)\in M\times \left[0,T\right] \,\,\, \middle|\,\,\, d(x,\tau)\leq R  \,\right\},
\end{equation*}
where $d(x,\tau)$ denotes the Riemannian distance from a fixed point induced from $g(\tau)$.
Then there exists a positive constant $C_{n}>0$ depending only on $n$ such that on $B_{R/2,T/4}$,
\begin{equation*}
\frac{\Vert \nabla u \Vert}{u}\leq C_n \left( \frac{1}{R}+\frac{1}{\sqrt{T}}+\sqrt{K}  \right) \left( 1+\log \frac{A}{u}  \right).
\end{equation*}
\end{thm}
\begin{proof}
We only sketch the proof.
We can prove it only by replacing the role of $\mathfrak{d}(x,\tau)$ with $d(x,\tau)$ in the definition of the cut-off function $\psi$ defined as (\ref{eq:cutoff function}) in Lemma \ref{prop:difficult}.
For such function $\psi$,
we need to present associated upper estimates for $\Psi_{1},\Psi_{2},\Psi_{3},\Psi_{4}$ defined as (\ref{eq:estimate constant}).

First,
$\Psi_1$ must be non-positive in virtue of the backward super Ricci flow inequality (\ref{eq:new back SRF}).
We discuss an upper estimate of $\Psi_{2}$.
Here we need the key techniques in \cite{BCP}.
Since the Ricci curvature is bounded from below in (\ref{eq:bounded Ricci}),
the Laplacian comparison implies
\begin{equation}\label{eq:Lap dtau}
\Delta d\leq (n-1)\sqrt{\frac{K}{n-1}}\, \frac{\cosh \sqrt{(n-1)^{-1}K}d}{\sinh \sqrt{(n-1)^{-1}K}d}\leq (n-1)\left( \frac{1}{d}+\sqrt{\frac{K}{n-1}} \right).
\end{equation}
Furthermore,
the backward super Ricci flow inequality (\ref{eq:new back SRF}) and the upper Ricci curvature bound in (\ref{eq:bounded Ricci}) lead us to
\begin{equation}\label{eq:time dtau}
\partial_{\tau} d= \int^{d}_{0}\,h(\gamma'(s),\gamma'(s))\,ds\leq \int^{d}_{0}\,\ric(\gamma'(s),\gamma'(s))\,ds\leq K d
\end{equation}
for a unit speed minimal geodesic $\gamma:[0,d]\to M$ from the fixed point (see e.g., Lemma 18.1 in \cite{CCGG2}).
We now use (\ref{eq:Lap dtau}), (\ref{eq:time dtau}) in (\ref{eq:different point}) instead of Lemma \ref{lem:reduced heat estimate2} in (\ref{eq:new different point}),
and $\Vert \nabla d \Vert=1$ in (\ref{eq:different point}) instead of Lemma \ref{lem:reduced gradient estimate} in (\ref{eq:new different point}),
which implies a similar upper estimate to (\ref{eq:estimate of A}).
Note that
due to the second term of the right hand side of (\ref{eq:Lap dtau}),
an additional term regarding $K/R^2$ appears in (\ref{eq:new estimate of A}).
We deal with such a term by dividing it into $K^2$-part and $1/R^4$-part with the help of the inequality of arithmetic-geometric means (cf. \cite{SZ}).

We also consider upper estimates of $\Psi_3,\Psi_{4}$.
In view of $\Vert \nabla d \Vert=1$,
we possess
\begin{equation}\label{eq:new grad est cutoff}
\frac{\Vert \nabla \psi \Vert^2}{\psi^{3/2}}\leq \frac{C^2_{3/4}}{R^2}
\end{equation}
instead of (\ref{eq:grad est cutoff}).
We obtain the desired estimates of $\Psi_3,\Psi_{4}$ by using (\ref{eq:new grad est cutoff}) in (\ref{eq:estimate of B}) and in (\ref{eq:estimate of C}) instead of (\ref{eq:grad est cutoff}),
respectively.
Based on these estimates,
we can prove the desired gradient estimate by the same argument as in the proof of Theorem \ref{thm:gradient estimate}.
\end{proof}

\begin{rem}\label{rem:BCP Liouville}
We can not conclude the associated Liouville theorem from Theorem \ref{thm:BCP} except for a quite specific case.
Actually,
to do so,
we need to apply Theorem \ref{thm:BCP} with $K=0$.
But in that case,
it must be Ricci flat by the assumption (\ref{eq:bounded Ricci}).
\end{rem}

\subsection{Zhang gradient estimate}
In the present subsection,
we are concerned with the work of Zhang \cite{Z}.
He proved a space-only local gradient estimate for positive solutions to heat equation along backward ancient Ricci flow (see Theorem 3.1 (a) in \cite{Z}).
We emphasize that
he has dealt with not Ricci flow but backward one.
In the same spirit as in Subsection \ref{sec:Bailesteanu-Cao-Pulemotov gradient estimate},
we can prove the following gradient estimate of Zhang type:
\begin{thm}[\cite{Z}]\label{thm:Z}
Let $(M,g(\tau))_{\tau \in [0,\infty)}$ denote an $n$-dimensional, complete sub Ricci flow,
namely, a subsolution to the Ricci flow equation $(\ref{eq:RF})$ defined by
\begin{equation}\label{eq:subRF}
\ric\leq -h.
\end{equation}
For $K\geq 0$,
we assume 
\begin{equation}\label{eq:ZRicci}
\ric\geq -Kg.
\end{equation}
Let $u:M \times [0,\infty) \to (0,\infty)$ be a positive solution to backward heat equation.
For $R,T>0$ and $A>0$,
we suppose $u\leq A$ on $B_{R,T}$.
Then there exists a positive constant $C_{n}>0$ depending only on $n$ such that on $B_{R/2,T/4}$,
\begin{equation*}
\frac{\Vert \nabla u \Vert}{u}\leq C_n \left( \frac{1}{R}+\frac{1}{\sqrt{T}}+\sqrt{K}  \right) \left( 1+\log \frac{A}{u}  \right).
\end{equation*}
\end{thm}
\begin{proof}
We only outline its proof.
In the same manner as Theorem \ref{thm:BCP},
one can show it only by considering the cut-off function $\psi$ determined by not $\mathfrak{d}(x,\tau)$ but $d(x,\tau)$ in Lemma \ref{prop:difficult}.
We explain the way to bound $\Psi_{1},\Psi_{2},\Psi_{3},\Psi_{4}$ defined as (\ref{eq:estimate constant}) in this setting.
We stress again that
we now deal with not backward Ricci flow but forward one unlike Theorem \ref{thm:BCP}.
Moreover,
we only have a lower Ricci curvature bound (\ref{eq:ZRicci}) unlike (\ref{eq:bounded Ricci}).

For $\Psi_{1}$,
the sub Ricci flow inequality (\ref{eq:subRF}),
the lower Ricci curvature bound (\ref{eq:ZRicci}),
the Young inequality (\ref{eq:Young}) with $p,q=2$,
and $\psi\leq 1$ yield
\begin{equation}\label{eq:Zestimate of C_1}
\Psi_1=-\frac{2\psi \mathcal{R}(\nabla f)}{(1-f)^2}\leq -\frac{4\psi \ric(\nabla f,\nabla f)}{(1-f)^2}\leq 4K \psi w \leq \epsilon \psi^2 w^2+\frac{4K^2}{\epsilon}\leq \epsilon \psi w^2+\frac{4K^2}{\epsilon},
\end{equation}
which corresponds to (\ref{eq:estimate of C_1}).
On $\Psi_{2}$,
the sub Ricci flow inequality (\ref{eq:subRF}) and the lower Ricci curvature bound in (\ref{eq:ZRicci}) also yield
\begin{equation}\label{eq:Ztime dtau}
\partial_{\tau} d= \int^{d}_{0}\,h(\gamma'(s),\gamma'(s))\,ds\leq -\int^{d}_{0}\,\ric(\gamma'(s),\gamma'(s))\,ds\leq K d
\end{equation}
for a unit speed minimal geodesic $\gamma:[0,d]\to M$ from the fixed point.
Since the conclusion of (\ref{eq:Ztime dtau}) is same as that of (\ref{eq:time dtau}),
this together with the Laplacian comparison (\ref{eq:Lap dtau}) yields the same upper estimate of $\Psi_2$ as in Theorem \ref{thm:BCP}.
Also,
we have the same upper estimates of $\Psi_3,\Psi_{4}$.
Thus,
we complete the proof by the same argument as in Theorem \ref{thm:gradient estimate}.
\end{proof}

Unlike Theorem \ref{thm:BCP},
one can conclude the following Liouville theorem by use of Theorem \ref{thm:Z} with $K=0$:
\begin{thm}[\cite{Z}]\label{thm:Zhang Liouville}
Let $(M,g(\tau))_{\tau \in [0,\infty)}$ be a complete sub Ricci flow of non-negative Ricci curvature.
Then we have the following:
\begin{enumerate}\setlength{\itemsep}{+0.7mm}
\item Let $u:M\times [0,\infty)\to (0,\infty)$ be a positive solution to backward heat equation.
If
\begin{equation*}
u(x,\tau)=\exp\left[o\left(d(x,\tau)+\sqrt{\tau}\right)\right]
\end{equation*}
near infinity, then $u$ is constant;
\item let $u:M\times [0,\infty)\to \mathbb{R}$ be a solution to backward heat equation.
If
\begin{equation*}
u(x,\tau)=o\left(d(x,\tau)+\sqrt{\tau}\right)
\end{equation*}
near infinity, then $u$ is constant.
\end{enumerate}
\end{thm}

\subsection{Ecker-Knopf-Ni-Topping gradient estimate}

We discuss the work of Ecker-Knopf-Ni-Topping \cite{EKNT}.
They have investigated a space-only local gradient estimate for positive solutions to conjugate heat equation along general geometric flow over positive time interval (see Theorem 10 in \cite{EKNT}).
Thanks to their method,
we can formulate the following gradient estimate of Ecker-Knopf-Ni-Topping in our setting and notation:
\begin{thm}[\cite{EKNT}]\label{thm:EKNT}
Let $(M,g(\tau))_{\tau \in [0,\infty)}$ be an $n$-dimensional, complete sub Ricci flow.
For $K,\overline{K} \geq 0$,
we assume 
\begin{equation}\label{eq:EKNTRicci}
\ric\geq -K g,\quad \Vert \nabla H \Vert\leq \overline{K}. 
\end{equation}
Let $u:M \times [0,\infty) \to (0,\infty)$ be a positive solution to conjugate heat equation
\begin{equation}\label{eq:conjugateHE}
\left(\Delta+\partial_{\tau}+H \right)u=0.
\end{equation}
For $R,T>0$ and $A>0$,
we suppose $u\leq A$ on $B_{R,T}$.
Then there exists a positive constant $C_{n}>0$ depending only on $n$ such that on $B_{R/2,T/4}$,
\begin{equation}\label{eq:EKNT}
\frac{\Vert \nabla u \Vert}{u}\leq C_n \left( \frac{1}{R}+\frac{1}{\sqrt{T}}+\sqrt{K}+\sqrt{\overline{K}}+\overline{K}^{1/4}  \right) \left( 1+\log \frac{A}{u}  \right).
\end{equation}
\end{thm}
\begin{proof}
Let us give an outline of the proof.
To begin with,
we need to prepare alternatives of Lemmas \ref{lem:easy} and \ref{lem:simple} for positive solution $u:M \times [0,\infty) \to (0,\infty)$ to conjugate heat equation (\ref{eq:conjugateHE}).
Under the same setting as in Lemma \ref{lem:easy},
we can verify
\begin{align}\label{eq:EKNTeasy1}
(\Delta+\partial_{\tau}) f&=-\Vert \nabla f \Vert^2-H,\\ \label{eq:EKNTeasy2}
(\Delta +\partial_{\tau}) \Vert \nabla f \Vert^2 & = 2\Vert \nabla^2 f \Vert^2-2g(\nabla \Vert \nabla f \Vert^2,\nabla f)+2\mathcal{R}(\nabla f)-2g(\nabla H,\nabla f)
\end{align}
by the same calculation as in the proof of Lemma \ref{lem:easy}.
Keeping (\ref{eq:EKNTeasy1}), (\ref{eq:EKNTeasy2}) in mind,
under the same setting as in Lemma \ref{lem:simple},
one can also deduce
\begin{equation}\label{eq:EKNTsimple}
\left(\Delta+\partial_{\tau}\right)w-\frac{2f g(\nabla w,\nabla f)}{1-f}\geq 2(1-f)w^2+\frac{2\mathcal{R}(\nabla f)}{(1-f)^2}-\frac{2g(\nabla H,\nabla f)}{(1-f)^2}-\frac{2H \Vert \nabla f \Vert^2}{(1-f)^3}
\end{equation}
from the same calculation as in the proof of Lemma \ref{lem:simple}.

Similarly to Theorem \ref{thm:Z},
we consider the cut-off function $\psi$ determined by not $\mathfrak{d}(x,\tau)$ but $d(x,\tau)$ in Lemma \ref{prop:difficult}.
In virtue of (\ref{eq:EKNTsimple}),
it holds that
\begin{equation*}
2(1-f)\psi w^2\leq \Psi_1+\Psi_2+\Psi_3+\Psi_4+\Psi_5+\Psi_6+\Phi
\end{equation*}
for $\Psi_{1},\Psi_{2},\Psi_{3},\Psi_{4}$ defined as (\ref{eq:estimate constant}),
for $\Psi_{5},\Psi_{6}$ defined by
\begin{equation*}
\Psi_5:=\frac{2\psi g(\nabla H,\nabla f)}{(1-f)^2},\quad \Psi_6:=\frac{2H\psi \Vert \nabla f \Vert^2}{(1-f)^3},
\end{equation*}
and for $\Phi$ defined as (\ref{eq:maximal constant}).
Since we assume the sub Ricci flow inequality (\ref{eq:subRF}) and the lower Ricci curvature bound in (\ref{eq:EKNTRicci}),
one can derive the same upper bounds for $\Psi_{1},\Psi_{2},\Psi_{3},\Psi_{4}$ as in Theorem \ref{thm:Z}.
Hence,
it suffices to examine upper bounds of $\Psi_{5},\Psi_{6}$.
By the Cauchy-Schwarz inequality,
the assumption (\ref{eq:EKNTRicci}) for $\Vert \nabla H \Vert$,
the facts $1/(1-f)\leq 1$ and $\psi \leq 1$,
the inequality of arithmetic-geometric means,
and the Young inequality (\ref{eq:Young}) with $p,q=2$,
\begin{align*}
\Psi_5&=\frac{2\psi g(\nabla H,\nabla f)}{(1-f)^2}\leq \frac{2\psi \Vert \nabla H \Vert \Vert \nabla f\Vert}{(1-f)^2}\leq  \frac{2\overline{K} \psi \Vert \nabla f\Vert}{1-f}\\
                       &\leq 2 \overline{K} (\psi w)^{1/2}\leq \overline{K} (1+\psi w)\leq \overline{K}+\epsilon \psi w^2+\frac{1}{4}\frac{\overline{K}^2}{\epsilon},
\end{align*}
where the first term in the right hand side of the above inequality will be reduced to $\overline{K}^{1/4}$-term in the right hand side of (\ref{eq:EKNT}).
Combining the sub Ricci flow inequality (\ref{eq:subRF}) and the lower Ricci curvature bound in (\ref{eq:EKNTRicci}),
we possess
\begin{equation*}
h\leq -\ric\leq K g,\quad H\leq nK,
\end{equation*}
and hence
\begin{equation*}
\Psi_6=\frac{2H\psi \Vert \nabla f \Vert^2}{(1-f)^3}\leq \frac{2nK \psi \Vert \nabla f \Vert^2}{(1-f)^3}\leq 2nK \psi w\leq \epsilon \psi w^2+\frac{n^2 K^2}{\epsilon}.
\end{equation*}
Here we used $1/(1-f)\leq 1$ and the Young inequality (\ref{eq:Young}) with $p,q=2$.
Summarizing these upper bounds,
we arrive at the desired gradient estimate.
\end{proof}

\begin{rem}\label{rem:EKNT Liouville}
Similarly to Theorem \ref{thm:BCP},
we can not deduce the associated Liouville theorem from Theorem \ref{thm:EKNT} except for a quite specific case.
For that purpose,
we use Theorem \ref{thm:BCP} with $K=0$ and $\overline{K}=0$.
Then,
however,
$H$ must be constant by the assumption (\ref{eq:EKNTRicci}).
\end{rem}

\subsection{Discussion}\label{sec:Discussion}
Based on the observation in the above three subsections,
we now describe our contributions comparing Theorem \ref{thm:gradient estimate} with other space-only local gradient estimates.

Due to the work of Zhang \cite{Z},
we see that
the Liouville theorem of Souplet-Zhang type holds for backward heat equation along sub Ricci flow under a growth condition for Riemannian distance (see Theorem \ref{thm:Zhang Liouville}).
Meanwhile,
in view of the work of Bailesteanu-Cao-Pulemotov \cite{BCP},
it seems to be difficult to yield such a Liouville theorem along backward super Ricci flow (see Theorem \ref{thm:BCP} and Remark \ref{rem:BCP Liouville}).
One of our contributions is to point out that
we can formulate the Liouville theorem by considering a growth condition with regard to reduced distance instead of Riemannian distance.

The reason why Theorem \ref{thm:BCP} does not imply the associated Liouville theorem is that
the upper Ricci curvature bound is assumed,
which leads to an upper bound of the time derivative of Riemannian distance (see (\ref{eq:time dtau})).
Note that
along sub Ricci flow,
such an upper bound is derived from the lower Ricci curvature bound (see (\ref{eq:Ztime dtau})).
In the proof of Theorem \ref{thm:gradient estimate},
we overcome this issue by using the time derivative formula (\ref{eq:time ell}) for reduced distance.

In the proof of Theorem \ref{thm:gradient estimate},
we need to control not only the time derivative for reduced distance but also its Laplacian.
In Theorems \ref{thm:BCP}, \ref{thm:Z}, \ref{thm:EKNT},
the Laplacian of Riemannian distance is controlled by the lower Ricci curvature bound (see (\ref{eq:Lap dtau})).
On the other hand,
in analyzing the Laplacian of reduced distance,
the lower Ricci curvature bound seems not to be compatible with our situation.
Our another contribution is to provide an insight that
the non-negative number $K$ appeared in the lower Ricci curvature bound in (\ref{eq:bounded Ricci}), (\ref{eq:ZRicci}), (\ref{eq:EKNTRicci}) should be included in the super Ricci flow inequality (i.e., $(-K)$-super Ricci flow inequality).
Then we can complete the proof by the Laplacian comparison (\ref{eq:Laplace ell}) for reduced distance. 

\subsection*{{\rm Acknowledgements}}
The first author was supported by JSPS KAKENHI (JP19K14521).
The second author was supported by JSPS Grant-in-Aid for Scientific Research on Innovative Areas ``Discrete Geometric Analysis for Materials Design" (17H06460).


\end{document}